\documentclass[11pt]{amsart}
\usepackage[margin=1.1in]{geometry}
\usepackage{amsmath,amssymb,amsthm}
\usepackage{array,booktabs,longtable}
\usepackage{xcolor}
\usepackage[colorlinks=true,linkcolor=blue!50!black,citecolor=blue!50!black]{hyperref}

\newif\iflinenumbers
\linenumbersfalse
\iflinenumbers
  \usepackage{lineno}
  \linenumbers
\fi

\newtheorem{theorem}{Theorem}[section]
\newtheorem{lemma}[theorem]{Lemma}
\newtheorem{proposition}[theorem]{Proposition}
\newtheorem{corollary}[theorem]{Corollary}
\theoremstyle{definition}
\newtheorem{definition}[theorem]{Definition}
\newtheorem{example}[theorem]{Example}
\theoremstyle{remark}
\newtheorem{remark}[theorem]{Remark}
\newtheorem{convention}[theorem]{Convention}

\DeclareMathOperator{\Inn}{Inn}
\DeclareMathOperator{\Aut}{Aut}
\DeclareMathOperator{\Out}{Out}
\DeclareMathOperator{\Syl}{Syl}
\DeclareMathOperator{\Fix}{Fix}
\DeclareMathOperator{\PSL}{PSL}
\DeclareMathOperator{\PSU}{PSU}
\DeclareMathOperator{\PSp}{PSp}
\DeclareMathOperator{\Sp}{Sp}
\DeclareMathOperator{\SL}{SL}
\DeclareMathOperator{\SU}{SU}
\DeclareMathOperator{\GL}{GL}
\DeclareMathOperator{\GU}{GU}
\DeclareMathOperator{\Sz}{Sz}
\DeclareMathOperator{\ord}{ord}
\newcommand{\Z}{\mathbb{Z}}
\newcommand{\F}{\mathbb{F}}
\newcommand{\Pos}{\mathcal{P}}
\newcommand{\vp}[1]{v_{#1}}
\newcolumntype{L}[1]{>{\raggedright\arraybackslash}p{#1}}

\title[Groups factorized by non-conjugate maximal subgroups]{Finite groups
that are the product of every pair of non-conjugate maximal subgroups are
soluble}

\author{Richie Sater}
\address{Independent Researcher, Washington, DC, United States}
\email{richiesater@gmail.com}
\urladdr{https://orcid.org/0009-0007-9051-8207}

\date{August 27, 2026}


\begin{document}

\begin{abstract}
We prove that every finite group that is the product of every pair of its
non-conjugate maximal subgroups is soluble, answering Problem 10.34 of the
Kourovka Notebook. The almost-simple case was proved by Tikhonenko and
Tyutyanov. We treat the remaining minimal-counterexample branch, where the
unique minimal normal subgroup has the form $S^k$, with $S$ nonabelian simple
and $k\geq2$. Two automorphism-stable coordinate subgroup classes satisfying a
maximal-supplement criterion produce non-conjugate maximal supplements. If the
factorization hypothesis held, their product would impose a $p$-adic
divisibility requirement growing linearly with $k$, while the quotient
contributes only coordinate outer automorphisms and a factor dividing $k!$. A
fixed valuation gap therefore excludes every $k\geq2$ at once. Suitable
subgroup classes are constructed uniformly across the infinite families of
finite simple groups using parabolic, torus, and primitive-prime-divisor
arguments; stable flag parabolics handle graph fusion, while GAP certificates
cover designated finite and sporadic cases. The resulting all-$k$ obstruction
provides a reusable mechanism for eliminating direct-power socles in
finite-group factorization problems.
\end{abstract}

\maketitle

\section{Introduction}\label{sec:intro}

This section states the problem and answer, identifies the new mechanism and
its inherited inputs, and gives a concise roadmap of the proof and its
verification boundary.

Problem 10.34 of the Kourovka Notebook, posed by V.~S.~Monakhov in 1986
\cite{Kourovka}, asks:

\begin{quote}
\emph{Does there exist a finite non-soluble group that coincides with the
product of any two of its non-conjugate maximal subgroups?}
\end{quote}

Many soluble groups have this property: it holds vacuously whenever all
maximal subgroups are conjugate, and $S_4=AB$ for every pair of non-conjugate
maximal subgroups $A,B$. The question is whether this factorization condition
forces solubility. It does.

\begin{theorem}[Main theorem]\label{thm:main}
Let $G$ be a finite group such that $G=AB$ for every pair of maximal
subgroups $A,B\leq G$ that are not conjugate in $G$. Then $G$ is soluble.
\end{theorem}

We say that $G$ has \emph{property $\mathrm{P}$} when it satisfies the
hypothesis of Theorem~\ref{thm:main}. Tikhonenko and Tyutyanov proved the
almost-simple case \cite{TT2010}; their result supplies the $k=1$ branch of
our proof. They also record that Zenkov had announced a negative answer in a
1997 conference abstract \cite{Zenkov1997}, while stating that no proof of the
announcement was known to them. The July 2026 Kourovka Notebook still lists
the problem as unsolved \cite{Kourovka}. Accordingly, this paper proves the
theorem but does not claim priority for the conclusion. After this paper's
arXiv v1 and v2, dated 4 August 2026 and 6 August 2026, respectively, Li and
Yang submitted a separately authored preprint dated 19 August 2026 that states
the same theorem \cite{LiYang2026}. Their proof uses product--socle lifting
from nonfactorizing almost-simple coordinate subgroups rather than the
stable-class valuation obstruction. These are distinct proof architectures.

For comparison, Theorem~\ref{thm:Dprime}
also recovers the almost-simple conclusion; it is not used in the proof of
Theorem~\ref{thm:main}.

The maximal factorizations of finite simple groups and their automorphism
groups were classified by Liebeck, Praeger, and Saxl \cite{LPS1990}. Those
tables identify the factorizations that occur. Here the decisive step is
instead a reusable obstruction: two local subgroup classes in a simple
coordinate group force an impossible divisibility condition in every
transitive direct-power extension. The classification of finite simple groups,
maximality and order data, automorphism descriptions, parabolic theory, and
Zsigmondy's theorem remain external inputs to the family coverage.

\subsection*{Proof mechanism}
Suppose that a minimal counterexample $G$ exists. Its unique minimal normal
subgroup has the form $N=S^k$, where $S$ is non-abelian simple, and $G$ has a
transitive coordinate realization inside an automorphism wreath product. If
$[V]$ is $X$-stable and self-normalizing, maximal in the stable-class poset
$\Pos_X(S)$, and normally saturating, then it produces a maximal supplement
$B_V=N_G(V^k)$ whose order is $|G/N||V|^k$.

If $[V]$ and $[W]$ are distinct suitable classes, property $\mathrm{P}$ forces
$G=B_VB_W$. The subgroup-product formula then imposes a $p$-adic divisibility
condition whose exponent grows linearly with $k$. By contrast, the $p$-part
available in $G/N$ is bounded by the coordinate outer automorphisms and $k!$.
A fixed valuation gap therefore gives a contradiction for every $k\geq2$.

It remains to construct one such pair for each non-abelian finite simple group.
Uniform parabolic, torus, and primitive-prime-divisor arguments handle the
infinite families. Graph automorphisms require stable flag-parabolic classes
in several families. GAP certificates handle the designated finite base,
sporadic groups, and exceptional parameters.

\subsection*{Contribution and dependence}
The conceptual hierarchy of the proof is summarized below. Imported
classification and primitive-prime-divisor results are inputs, not claims of
novelty.

\begin{table}[htbp]
\small
\begin{tabular}{@{}p{0.42\textwidth}p{0.50\textwidth}@{}}
\toprule
Component & Status \\
\midrule
Almost-simple $k=1$ result & Prior theorem of Tikhonenko--Tyutyanov \\
CFSG and maximal-subgroup/order data & Cited external inputs \\
Reduction to a unique socle $S^k$ & Standard ingredients assembled for property $\mathrm{P}$ \\
Suitable stable, self-normalizing classes satisfying poset maximality and normal saturation & Central maximal-supplement machinery with exact order \\
Uniform-in-$k$ product-supplement valuation obstruction & Principal conceptual result \\
Stable flag parabolics under graph fusion & Distinctive structural mechanism \\
Infinite-family coverage & Application of the criterion using cited classification and arithmetic inputs \\
Finite and sporadic coverage & GAP-certified component \\
Lean and Rocq developments & Partial, conditional formal verification; not an end-to-end proof \\
\bottomrule
\end{tabular}
\caption{New, inherited, and externally supplied components.}
\label{tab:contributions}
\end{table}

The mathematical yield is the passage from local stable-class data in $S$ to
an obstruction valid for every multiplicity $k$. Poset maximality and normal
saturation are the structural assumptions that turn local classes into
maximal supplements; the valuation gap is the transferable numerical input.
The graph-fusion cases show that ordinary maximal subgroups are not essential:
stable flag parabolics can replace fused maximal classes.

The manuscript proves the theorem by ordinary mathematical argument. Lean
checks named structural and arithmetic statements under explicit hypotheses;
Rocq/MathComp supplies a direct-power component through an audited interface;
GAP certifies designated finite cases; and Python checks manifests,
certificates, source maps, and independent arithmetic. Appendix~\ref{app:trust}
states exactly how these components compose and which claims remain external.

\section{Reduction to the monolithic coordinate case}\label{sec:reduction}

By the end of this section, every counterexample has been converted into a
transitive $S^k$-coordinate problem. The remaining task will be to construct
two incompatible maximal supplements from local subgroup classes of $S$.
The quotient argument and the monolithic reduction are standard; the explicit
coordinate closure records exactly the arithmetic contribution available
above the socle.

\begin{lemma}[Factorizations and conjugate representatives]\label{lem:conj}
If $G=AB$ for subgroups $A,B$, then $G=A^xB^y$ for all $x,y\in G$.
\end{lemma}

\begin{proof}
Since $G=AB$, write $xy^{-1}=ab$ with $a\in A$, $b\in B$. Then, as
subsets of $G$,
\[
A^xB^y=x^{-1}Ax\,y^{-1}By=x^{-1}A(ab)By
 =x^{-1}(AB)\,y=x^{-1}Gy=G. \qedhere
\]
\end{proof}

Consequently, property $\mathrm{P}$ depends only on the conjugacy classes
of maximal subgroups: it says every pair of \emph{distinct classes} has
one (hence every) representative pair factorizing $G$.

\begin{lemma}\label{lem:quot}
Property $\mathrm{P}$ is inherited by quotients.
\end{lemma}

\begin{proof}
Maximal subgroups of $G/K$ are the images of the maximal subgroups of $G$
containing $K$; non-conjugate ones lift to non-conjugate ones, and
factorizations project.
\end{proof}

\begin{proposition}[Minimal counterexample structure]\label{prop:min}
Let $G$ have least order among the non-soluble groups with property
$\mathrm{P}$.
Then $G$ has a unique minimal normal subgroup $N=S^k$ with $S$ non-abelian
simple and $k\geq 1$, $C_G(N)=1$, $G/N$ is soluble, and $G$ embeds in
$\Aut(S)\wr S_k$ with $N$ mapping onto $\Inn(S)^k$ and $G$ acting
transitively on the $k$ coordinates.
\end{proposition}

\begin{proof}
Every proper quotient of $G$ has property $\mathrm{P}$
(Lemma~\ref{lem:quot}), hence is soluble by minimality of $|G|$. The
soluble radical $R$ of $G$ is trivial: otherwise $G/R$ is non-soluble
(as $R$ is soluble) with property $\mathrm{P}$, contradicting minimality.
So the socle of $G$ is a direct product of non-abelian minimal normal
subgroups. If $N_1\neq N_2$ were two of them, then $G/N_1$ is soluble
while $N_2\cong N_2N_1/N_1\leq G/N_1$ is non-soluble, a contradiction.
Hence there is a unique minimal normal subgroup $N=S^k$, $S$ non-abelian
simple. Now $C_G(N)\cap N=Z(N)=1$ since $N$ is a direct product of
non-abelian simple groups; if $C_G(N)$ were nontrivial it would contain
a minimal normal subgroup of $G$, necessarily $N$ by uniqueness,
contradicting $C_G(N)\cap N=1$. So $C_G(N)=1$, and the conjugation
action embeds $G$ into $\Aut(N)=\Aut(S)\wr S_k$ in the standard way;
minimal normality of $N$ forces transitivity on coordinates. Published
references for the two standard
structural inputs used here are \cite[Lemma~2.10, pp.~173--174]{ZhangShi2009}
for the automorphism group of a direct power and
\cite[p.~432]{LucchiniMorini2002} for the resulting transitive wreath
realization in the unique-minimal-normal-subgroup setting. Finally, $G/N$ is
soluble because it is a proper quotient.
\end{proof}

\begin{definition}[Coordinate closure]\label{def:X}
For $G$ as in Proposition~\ref{prop:min} let $G_1$ be the stabilizer of
the first coordinate and $\pi_1\colon G_1\to\Aut(S)$ the projection. Set
$X:=\pi_1(G_1)\cdot\Inn(S)$, the \emph{coordinate closure}. By
coordinate-transitivity all coordinates give the same $X$ up to conjugacy,
and $\Inn(S)\leq X\leq\Aut(S)$.
\end{definition}

\begin{lemma}[Coordinate normalization]\label{lem:coordinate-normalization}
With $X$ as in Definition~\ref{def:X}, the wreath realization can be
conjugated inside $\Aut(S)\wr S_k$ so that
\[
 N\leq G\leq X\wr S_k.
\]
Every component automorphism of every element of the conjugated copy of $G$
then lies in $X$.
\end{lemma}

\begin{proof}
Use the left-conjugation maps $c_g(z)=gzg^{-1}$, while retaining
$A^g=g^{-1}Ag$ for subgroup conjugates. If an element sends $S_i$ to $S_j$,
its $i\!\to\!j$ component is the induced automorphism under the fixed
identifications $S_i\cong S\cong S_j$; compositions of these automorphisms
are read in the ordinary right-to-left order. Fix $g_j\in G$ carrying
coordinate $1$ to coordinate $j$, with $g_1=1$, and let $a_j\in\Aut(S)$ be
its $1\!\to\!j$ component. If $g\in G$ carries coordinate $i$ to coordinate
$j$ through the component $c\in\Aut(S)$, then $g_j^{-1}\,g\,g_i$ stabilizes
coordinate $1$, because its left-conjugation map follows
\[
 S_1\xrightarrow{\ c_{g_i}\ }S_i
 \xrightarrow{\ c_g\ }S_j
 \xrightarrow{\ c_{g_j^{-1}}\ }S_1.
\]
Its component is $a_j^{-1}\circ c\circ a_i$, abbreviated
$a_j^{-1}ca_i$, and therefore lies in $\pi_1(G_1)\leq X$. Put
$\delta:=(a_1,\dots,a_k)\in\Aut(S)^k$. For an element with
$i\!\to\!j$ component $c$, the corresponding component after replacing
$G$ by $\delta^{-1}G\delta$ is exactly $a_j^{-1}ca_i\in X$. Thus every
component of every element of the conjugated group lies in $X$, while
$a_1=1$ leaves $X$ itself unchanged. This is the required normalized
realization.
\end{proof}

\begin{convention}[Normalized coordinate model]\label{conv:X}
Henceforth $G$ denotes the conjugate supplied by
Lemma~\ref{lem:coordinate-normalization}. Property $\mathrm{P}$ and all
criteria below are invariant under this replacement. They are also invariant
under replacing $G$ by a conjugate under
$(a,\dots,a)\in\Aut(S)^k$, so statements may be proved after normalizing
$X$ up to $\Aut(S)$-conjugacy.
\end{convention}

\begin{convention}[Coordinate action and arithmetic
notation]\label{conv:coord}
For \emph{coordinate} calculations we fix the following convention.
Write $c_g(z):=gzg^{-1}$ (left conjugation), let $S_i$ be the $i$-th
coordinate factor of $N\cong S^k$, and let $\sigma_g\in S_k$ be the
permutation determined by $c_g(S_j)=S_{\sigma_g(j)}$. There are
component automorphisms $a_{g,i}\in\Aut(S)$ such that
\[
\bigl(c_g(z)\bigr)_i=a_{g,i}\bigl(z_{\sigma_g^{-1}(i)}\bigr)
\qquad(z\in N,\ 1\leq i\leq k):
\]
the value arriving at coordinate $i$ comes from the
\emph{inverse-permuted} source coordinate $\sigma_g^{-1}(i)$. From
$c_{gh}=c_g\circ c_h$ these data compose by
\[
\sigma_{gh}=\sigma_g\sigma_h,\qquad
a_{gh,i}=a_{g,i}\,a_{h,\sigma_g^{-1}(i)} .
\]
After the normalization of Convention~\ref{conv:X}, every $a_{g,i}$
lies in $X$. Throughout, $c_g$ denotes left conjugation on elements, whereas
$A^g=g^{-1}Ag$ denotes conjugation of subgroups; thus
$A^g=c_{g^{-1}}(A)$. Write
\[
x:=|X/\Inn(S)|,\qquad t:=|G/N| .
\]
\end{convention}

\begin{lemma}[Wreath-top quotient divisor]\label{lem:wreath-divisor}
With the normalized coordinate model and notation above,
\[
 t\mid x^k k!.
\]
\end{lemma}

\begin{proof}
The quotient map $X\to X/\Inn(S)$ induces
\[
 \theta\colon X\wr S_k\longrightarrow (X/\Inn(S))\wr S_k
\]
with kernel $\Inn(S)^k$. Under the conjugation embedding of
Proposition~\ref{prop:min}, the subgroup $N=S^k$ acts as
$\Inn(S)^k$, so $N\leq G\cap\ker\theta$. Conversely, if
$g\in G\cap\Inn(S)^k$, choose $n\in N$ inducing the same tuple of inner
automorphisms on $N$. Then $gn^{-1}\in C_G(N)=1$ by
Proposition~\ref{prop:min}, and hence $g=n$. Thus
$G\cap\ker\theta=N$, and $G/N$ embeds in
$(X/\Inn(S))\wr S_k$. Lagrange's theorem gives
\[
 |G/N|\mid |(X/\Inn(S))\wr S_k|=x^k k!,
\]
as required.
\end{proof}

The following observation is not needed for the proof of
Theorem~\ref{thm:main} but explains why the pairs used later are of
``supplement'' type.

\begin{remark}[Top--supplement pairs always factorize]
If $M\supseteq N$ is maximal and $B$ is maximal with $B\not\supseteq N$,
then $BN=G$ ($BN$ is a subgroup strictly containing the maximal $B$),
and since $N\leq M$ we get $MB\supseteq NB=G$. Hence $G$ has property
$\mathrm{P}$ if and only if $G/N$ has property $\mathrm{P}$ and the product
of every pair of non-conjugate maximal subgroups not containing $N$ is $G$.
\end{remark}

\section{Stable classes and maximal product supplements}\label{sec:machinery}

This section identifies the precise hypotheses under which a local class in
$S$ yields a maximal subgroup of every admissible $S^k$-extension. Stability
and self-normalization give the supplement, intersection, and exact order;
poset maximality and normal saturation give maximality. The nonconjugacy
statement and graph-fusion substitute complete the engine recorded in
Corollary~\ref{cor:supplement-engine}.

Fix $G$ as in Conventions~\ref{conv:X} and~\ref{conv:coord}:
$N=S^k\leq G\leq X\wr S_k$,
$t=|G/N|$, $x=|X/\Inn(S)|$. For $V\leq S$ write $[V]$ for its
$S$-conjugacy class. An $S$-class is \emph{$X$-stable} if it is invariant
under the action of $X$,
i.e.\ $V^a$ is $S$-conjugate to $V$ for every $a\in X$ ($\Inn$-stability
being automatic).

\begin{definition}\label{def:poset}
$\Pos_X(S):=\{\,[V] : 1<V<S,\ N_S(V)=V,\ [V]\ X\text{-stable}\,\}$,
partially ordered by containment up to $S$-conjugacy
($[V]\leq[U]$ iff $V\leq U^s$ for some $s\in S$). Call $V$
\emph{normally saturating} if $\langle V^W\rangle=W$ for every overgroup
$V\leq W\leq S$.
\end{definition}

\begin{remark}\label{rem:maxauto}
If $V$ is a \emph{maximal} subgroup of $S$ and $[V]$ is $X$-stable, then
$[V]\in\Pos_X(S)$ (maximal subgroups of a non-abelian simple group are
self-normalizing), $[V]$ is a maximal element of the poset, and $V$ is
normally saturating (its only overgroups are $V$ and $S$, and
$\langle V^S\rangle=S$ by simplicity). Thus the lemmas below apply to
maximal classes with no further hypotheses.
\end{remark}

\begin{lemma}[Product-supplement construction]\label{lem:A}
Let $[V]\in\Pos_X(S)$ and set $B_V:=N_G(V^k)$. Then
\begin{enumerate}
\item $B_V\cap N=N_S(V)^k=V^k$;
\item $B_VN=G$;
\item $|B_V|=t\cdot|V|^k$.
\end{enumerate}
\end{lemma}

\begin{proof}
(1) An element of $N=S^k$ normalizes $V^k$ iff each coordinate normalizes
$V$. (2) Frattini-type argument: let $g\in G$. In the coordinate-action
notation of Convention~\ref{conv:coord}, $(V^k)^g=c_{g^{-1}}(V^k)$, and
since every coordinate of $V^k$ is $V$,
\[
(V^k)^g=\prod_{i=1}^{k} a_{g^{-1},i}(V),
\]
each factor is the image of $V$ under a component automorphism
$a_{g^{-1},i}\in X$. By $X$-stability there are $s_i\in S$ with
$a_{g^{-1},i}(V)=V^{s_i}$. So $(V^k)^g$ is $N$-conjugate to
$V^k$: there is $n\in N$ with $(V^k)^{gn}=V^k$, i.e.\ $gn\in B_V$, so
$g\in B_VN$. (3) $|B_V|=|B_VN|\cdot|B_V\cap N|/|N|=t\,|V|^k$.
\end{proof}

\begin{lemma}[Maximality of stable product normalizers]\label{lem:B}
If in addition $[V]$ is a maximal element of $\Pos_X(S)$ and $V$ is
normally saturating, then $B_V$ is maximal in $G$.
\end{lemma}

\begin{proof}
Let $B_V\leq H\leq G$. Since $H\supseteq B_V$ and $B_VN=G$ we have
$HN=G$, so $H$ surjects onto $G/N$ and in particular acts
coordinate-transitively.

\emph{Step 1 (the coordinate components of $H$ recover $X$).} Let $H_1,G_1$ be the
coordinate-$1$ stabilizers. For $g\in G_1$ write $g=hn$ ($h\in H$,
$n\in N$); $n$ fixes every coordinate, hence so does $h$, so $h\in H_1$
and $\pi_1(G_1)=\pi_1(H_1)\cdot\pi_1(N\cap G_1)=\pi_1(H_1)\cdot\Inn(S)$.
Therefore $\pi_1(H_1)$ and $\Inn(S)$ together generate $X$, and a class
of subgroups of $S$ is $X$-stable iff it is stable under $\pi_1(H_1)$.

\emph{Step 2 (coordinate kernels and saturation: $H\cap N$ is a product).} Let
$T:=H\cap N\supseteq V^k$, let $W_i:=\pi_i(T)\leq S$ be the coordinate
projections and $K_i:=T\cap S_i$ the coordinate kernels ($S_i$ the $i$-th
factor). Then $V\leq K_i$ (as $V^k\cap S_i=V$ in coordinate $i$),
$K_i\trianglelefteq W_i$, and $V\leq W_i$. Since $K_i$ is a normal
subgroup of $W_i$ containing $V$, it contains
$\langle V^{W_i}\rangle=W_i$ by saturation. So $K_i=W_i$ for every $i$
and $T=\prod_i W_i$.

\emph{Step 3 (the $W_i$ force $H=B_V$ or $H=G$).} $T$ is $H$-invariant.
For $h\in H_1$, projecting $T^h=T$ to the first coordinate gives
$W_1^{\pi_1(h)}=W_1$, so $[W_1]$ is $\pi_1(H_1)$-stable, hence
$X$-stable by Step~1. For any $i$, coordinate-transitivity of $H$
supplies $h\in H$ carrying coordinate $i$ to coordinate $1$ through a
component $c\in X$, and projecting $T^h=T$ gives $W_1=W_i^{\,c}$;
applying $X$-stability of $[W_1]$ to $c^{-1}$ gives $[W_i]=[W_1]$.
Write $[W]$ for this common $X$-stable class.
If $W=S$ then $T=N$, so $H\supseteq N$ and $H=HN=G$.
Otherwise $1<V\leq W<S$. Form the normalizer tower
$W\leq N_S(W)\leq N_S(N_S(W))\leq\cdots$; it is strictly increasing
until it becomes self-normalizing, and it cannot reach $S$ (the
penultimate term would be a proper normal subgroup of the simple group
$S$). Its terminal member $U$ is proper, self-normalizing, and $[U]$ is
$X$-stable ($N_S(W^a)=N_S(W)^a$, so stability propagates up the tower),
i.e.\ $[U]\in\Pos_X(S)$ with $[U]\geq[W]\geq[V]$. Maximality of $[V]$ in
$\Pos_X(S)$ forces $[U]=[V]$; then $|U|=|V|$ and $V\leq W\leq U$ gives
$V=W=U$. Since $V\leq W_i$ by Step~2 and $[W_i]=[W]$ gives
$|W_i|=|W|=|V|$, each $W_i=V$; so $T=\prod_i W_i=V^k$, hence
$|H|=|HN|\cdot|H\cap N|/|N|=t\,|V|^k=|B_V|$
and $H=B_V$.
\end{proof}

\begin{lemma}[Nonconjugacy of distinct stable classes]\label{lem:C}
Distinct $[V]\neq[W]$ in $\Pos_X(S)$ give non-conjugate $B_V,B_W$.
\end{lemma}

\begin{proof}
If $B_V^g=B_W$ then $(B_V\cap N)^g=B_W\cap N$ (as $N\trianglelefteq G$),
i.e.\ $(V^k)^g=W^k$. Reading off any coordinate, $W$ is the image of $V$
under an element of $X$ composed with a coordinate permutation, so
$[W]=[V^a]=[V]$ for some $a\in X$ by $X$-stability, a contradiction.
\end{proof}

The next lemma supplies, structurally, all hypotheses of
Lemma~\ref{lem:B} for the non-maximal flag-parabolic substitutes used in
Appendix~\ref{app:families}: the standard parabolic subgroups whose maximal
parabolic overgroups are fused by a graph automorphism. We use the
notation of \cite{Carter}. Here $S$ is an untwisted simple group of Lie
type with a split $BN$-pair of rank $\ell\geq2$; all our applications
are untwisted. Write $B=U\rtimes T$ for a Borel subgroup, with unipotent
radical $U$ and maximal torus $T$, and let $\Delta$ be the node set of
the Dynkin diagram. For $J\subseteq\Delta$, let $Q_J\supseteq B$ be the
standard parabolic whose Levi subgroup $L_J$ is generated by $T$ and the
root subgroups corresponding to $J$. The maximal parabolic omitting node
$i$ is $P_i=Q_{\Delta\setminus\{i\}}$.

\begin{lemma}[Flag-parabolic stability under graph fusion]\label{lem:P}
Fix $J\subseteq\Delta$ and put $V:=Q_J$. Suppose $X$ is such that
\begin{enumerate}
\item[(a)] the $S$-class $[Q_J]$ is $X$-stable, and
\item[(b)] for every $J\subsetneq K\subsetneq\Delta$ the class $[Q_K]$
is \emph{not} $X$-stable.
\end{enumerate}
Then $[V]\in\Pos_X(S)$, $[V]$ is a maximal element of $\Pos_X(S)$, and
$V$ is normally saturating. Consequently Lemmas~\ref{lem:A}--\ref{lem:C}
apply to $V$ with no computational hypothesis.
\end{lemma}

\begin{proof}
Parabolic subgroups are self-normalizing
\cite[Thm.~8.3.3, p.~112]{Carter}, so
$[V]\in\Pos_X(S)$ by (a).

\emph{Overgroups.} Every subgroup containing $V$ contains $B$, and every
subgroup containing a Borel subgroup is a parabolic $Q_K$, $K\supseteq J$,
of the same $BN$-pair \cite[Thm.~8.3.2, p.~112]{Carter}; moreover $Q_K$
and $Q_{K'}$ are $S$-conjugate iff $K=K'$
\cite[Thm.~8.3.3, p.~112]{Carter}.

\emph{Poset-maximality.} If $[U]\in\Pos_X(S)$ with $V\leq U^s<S$ for some
$s$, then $U^s=Q_K$ for some $J\subseteq K\subsetneq\Delta$; $K=J$ gives
$[U]=[V]$, while $J\subsetneq K$ contradicts (b), since $[U]=[Q_K]$ would
be $X$-stable.

\emph{Saturation.} Let $V\leq W\leq S$, so $W=Q_K$ ($K\supseteq J$) or
$W=S$; we must show $\langle V^W\rangle=W$. The normal closure of $V$ in
$W$ contains that of $B$. Write $W=U_K\rtimes L_K$ (Levi decomposition;
for $W=S$ read $U_K=1$, $L_K=S$). Then $U_K\leq B$, and modulo $U_K$ the
image of $B$ is a Borel subgroup $B_L=T\cdot(U\cap L_K)$ of $L_K$
containing the full torus $T$. For each node $\alpha\in K$ the
reflection representative $n_\alpha$ lies in $L_K$, and
$U_{-\alpha}=U_\alpha^{\,n_\alpha}$. The subgroup
$\langle B_L^{\,L_K}\rangle$ contains every $L_K$-conjugate of every
subgroup of $B_L$; in particular it contains $T$, each $U_\alpha$ for
$\alpha\in K$, and each $U_{-\alpha}=U_\alpha^{\,n_\alpha}$.
These generate $L_K$ by the root-subgroup description of the Levi
decomposition \cite[\S8.5, pp.~118--120]{Carter}. So
$\langle B^W\rangle\supseteq U_K\cdot L_K=W$.
\end{proof}

\begin{corollary}[Stable product-supplement engine]
\label{cor:supplement-engine}
Let $[V]\neq[W]$ be maximal elements of $\Pos_X(S)$ and suppose that both
$V$ and $W$ are normally saturating. Then
\[
 B_V=N_G(V^k),\qquad B_W=N_G(W^k)
\]
are non-conjugate maximal supplements of $N$, with
\[
 B_V\cap N=V^k,\quad B_W\cap N=W^k,\quad
 |B_V|=t|V|^k,\quad |B_W|=t|W|^k.
\]
Either or both of the selected classes may instead be represented by a flag
parabolic satisfying the hypotheses of Lemma~\ref{lem:P}; the same conclusion
then holds for the resulting pair.
\end{corollary}

\begin{proof}
Apply Lemmas~\ref{lem:A}--\ref{lem:C}; Lemma~\ref{lem:P} supplies the
stable-poset maximality and normal-saturation hypotheses in the flag case.
\end{proof}

The corollary isolates the exact output needed below: two non-conjugate
maximal supplements and their orders. No family-specific classification data
enters this engine.

\section{The stable-class divisibility criterion}\label{sec:criterion}

The purpose of this section is to convert the supplement engine into the
uniform obstruction that drives the paper. The key comparison is between a
missing $p$-part repeated in every socle coordinate and the smaller $p$-budget
available in the wreath top.

For a prime $p$, $\vp{p}$ denotes the $p$-adic valuation and $s_p(k)$ the
digit sum of $k$ in base $p$, so that $\vp{p}(k!)=(k-s_p(k))/(p-1)$.

\begin{theorem}[Uniform stable-class divisibility criterion]\label{thm:D}
Let $S$ be non-abelian simple, let
$\Inn(S)\leq X\leq\Aut(S)$, and put $x:=|X/\Inn(S)|$. Let
$[V]\neq[W]$ be maximal elements of $\Pos_X(S)$, both normally saturating.
If some prime $p$ satisfies
\[
\eta=\eta_p(V,W):=\vp{p}(|S|)-\vp{p}(|V|)-\vp{p}(|W|)\;>\;\vp{p}(x),
\]
then, for every $k\geq2$, no finite group $G$ with unique minimal normal
subgroup $N\cong S^k$, whose action on the factors of $N$ has coordinate
closure $X$, has property $\mathrm{P}$.
\end{theorem}

\begin{proof}
Conjugate the wreath realization as in
Lemma~\ref{lem:coordinate-normalization}; then
Conventions~\ref{conv:X}--\ref{conv:coord} and
Lemmas~\ref{lem:A}--\ref{lem:C} apply.
In particular, $B_V$ and $B_W$ are non-conjugate
maximal subgroups of $G$, so property $\mathrm{P}$ forces $G=B_VB_W$ and
\[
|B_V\cap B_W|=\frac{|B_V|\,|B_W|}{|G|}
 = t\cdot\Bigl(\frac{|V|\,|W|}{|S|}\Bigr)^{k}\in\Z ,
\]
whence $p^{\eta k}\mid t$. But
Lemma~\ref{lem:wreath-divisor} gives
\[
\vp{p}(t)\;\leq\;k\,\vp{p}(x)+\vp{p}(k!)
 \;=\;k\,\vp{p}(x)+\frac{k-s_p(k)}{p-1}
 \;<\;k\bigl(\vp{p}(x)+1\bigr)\;\leq\;\eta k ,
\]
a contradiction.
\end{proof}

\begin{example}[The criterion for $A_5$]\label{ex:a5}
Let $S=A_5$. For every $\Inn(S)\leq X\leq\Aut(S)$, the classes of
$V=A_4$ and $W\cong S_3$ are $X$-stable: they are the unique classes of
maximal subgroups of orders $12$ and $6$, respectively. Thus they satisfy the
structural hypotheses of Theorem~\ref{thm:D}. With $p=5$,
\[
 \eta_5(V,W)=\vp{5}(60)-\vp{5}(12)-\vp{5}(6)=1.
\]
Since $|X/\Inn(S)|$ divides $|\Out(A_5)|=2$, its $5$-valuation is zero.
Theorem~\ref{thm:D} therefore excludes every admissible group with socle
$A_5^k$ at once, for all $k\geq2$. This small example displays the complete
workflow: stable maximal classes, one obstruction prime, one strict gap, and
no enumeration over $k$.
\end{example}

\begin{remark}\label{rem:crit}
(i) The criterion is uniform in $k$, and this uniformity handles the
infinite families. Subgroup sizes alone cannot yield a uniform contradiction:
using $t\leq x^k k!$, the natural upper bound for
$|B_V||B_W|/|G|$ is
\[
 \left(x\frac{|V||W|}{|S|}\right)^k k!,
\]
whose $k$-th root is
$x|V||W|\,(k!)^{1/k}/|S|$ and therefore diverges. The valuation gap is
essential. (ii) Only the \emph{existence} of
$B_V,B_W$ is used; no classification of the maximal subgroups of $G$
(or even of $S$) is needed. In particular the criterion is insensitive
to completeness of tabulated maximal-subgroup data: it needs only that
the two exhibited subgroups are maximal, or are flag-parabolic substitutes covered by
Lemma~\ref{lem:P}, and that their classes are stable. (Machine
\emph{verification} of stability for
the finite base does use enumerated subgroup data; see the trust
boundary in Appendix~\ref{app:trust}.)
\end{remark}

A related monolithic-coordinate construction appears in the study of finite
groups whose maximal subgroups have exact complements \cite{Sater1868}.
There the obstruction is not two-supplement integrality: a coordinate action
is lifted to a primitive product action, and the absence of a regular subgroup
excludes a complement. The wreath-top valuation budget is shared preliminary
machinery for two otherwise different arguments; neither proof reduces to the
other criterion.

The criterion has a $k=1$ counterpart, which recovers the almost
simple case.

\begin{theorem}[$k=1$: the almost simple case]\label{thm:Dprime}
Let $S\leq G\leq\Aut(S)$ with $S$ non-abelian simple, and put $X:=G$,
so that $x:=|G/S|=t$. Let $[V]\neq[W]$ be maximal elements of
$\Pos_X(S)$, both normally saturating, and set $B_V:=N_G(V)$,
$B_W:=N_G(W)$. Then:
\begin{enumerate}
\item $B_V\cap S=V$, $B_VS=G$, and $|B_V|=t|V|$;
\item $B_V$ and $B_W$ are non-conjugate maximal subgroups of $G$;
\item if $\eta=\vp{p}(|S|)-\vp{p}(|V|)-\vp{p}(|W|)>\vp{p}(x)$ for some
prime $p$, then $G$ does not have property $\mathrm{P}$.
\end{enumerate}
\end{theorem}

The proof is recorded in Appendix~\ref{app:almost-simple}. It is retained as
a comparison with the $k\geq2$ mechanism and is not used in the proof of the
main theorem.

\begin{convention}\label{conv:excluded}
From now on, ``$S$ is \emph{excluded}'' means: for every $X$ with
$\Inn(S)\leq X\leq\Aut(S)$ there exist classes as in
Theorem~\ref{thm:D}; consequently no finite group $G$ with unique
minimal normal subgroup $S^k$ ($k\geq2$, any admissible embedding) has
property $\mathrm{P}$. The almost simple branch $k=1$ is supplied separately
by Tikhonenko--Tyutyanov \cite{TT2010}.
\end{convention}

\section{Two representative applications}\label{sec:applications}

The criterion is most useful when the family work is reduced to a repeatable
sequence: select stable classes, prove maximality, choose an obstruction prime,
check the group and subgroup valuations, bound the outer contribution, and
route the exceptions. We retain one ordinary application and one graph-fusion
application here. Appendix~\ref{app:families} gives the remaining family
proofs in the same format.

\subsection{Ordinary stable maximal classes: alternating groups}

\begin{theorem}\label{thm:an}
$S=A_n$, $n\geq5$, is excluded.
\end{theorem}

\begin{proof}
\emph{Finite exceptions.} For $n\leq14$ this is covered by finite GAP certificates: the cases $n\leq10$ fall under
Proposition~\ref{prop:base} (including $n=6$, whose exceptional outer
automorphisms fuse classes; all five $X$ are handled there), and
the cases $11\leq n\leq14$ are covered by the finite certificates of
Appendix~\ref{app:finite}.
Assume $n\geq15$, so $\Aut(A_n)=S_n$ and $x\mid2$.

\emph{Obstruction prime.} Let $p$ be the largest prime $\leq n$. We first record the bound
$p\geq n/2+3$. By Nagura's theorem
\cite[Theorem, pp.~180--181]{Nagura}, there is a prime in $(y,6y/5]$
for $y\geq25$. Thus, for $n\geq31$, there is a prime in $(5n/6,n]$,
and $5n/6>n/2+3$ for $n>18$. For $15\leq n\leq30$, the largest prime
$p\leq n$ satisfies the bound by inspection:
$p=13,13,17,17,19,19,19,19,23,23,23,23,23,23,29,29$ for
$n=15,\dots,30$. In particular, $n/2<p\leq n$.

\emph{Selected classes, maximality, and stability.} For $m<n/2$ let
$M_m:=(S_m\times S_{n-m})\cap A_n$, and for $n$ even let
$I:=(S_{n/2}\wr S_2)\cap A_n$. Each $M_m$ is maximal in $A_n$, and $I$
is maximal except in a finite list of cases, all with $n\leq14$
\cite[Thm.~1 and Tables~I--II, pp.~366--368]{LPS1987}.
Conjugation by $S_n$ preserves the partition shape
(resp.\ block structure), each shape is a single $A_n$-class, and
$\Aut(A_n)=S_n$ for $n\neq6$: all these classes are $X$-stable for every
$X$.

\emph{Occurrence, avoidance, and outer contribution.}
For any two distinct classes among
$\{M_m:\,n-p<m<n/2\}\cup\{I\ \text{if $n$ even and}\ n/2<p\}$, we have
$\vp{p}(|A_n|)=1$, while $\vp{p}$ of each
subgroup order is $0$ (all parts of the partitions involved are $<p$),
and $\vp{p}(x)\leq\vp{p}(2)=0$. Thus $\eta=1>0$, and
Theorem~\ref{thm:D} applies, provided two such classes
exist. The number of admissible intransitive $m$ is
$p-\lfloor n/2\rfloor-1$, plus one for $I$ when $n$ is even; two classes
exist because $p\geq n/2+3$.

\emph{Conclusion.} Theorem~\ref{thm:D} excludes $A_n$ for every
$n\geq15$, while the exact finite certificates cover the preceding degrees.
\end{proof}

The alternating family illustrates the ordinary mechanism: maximality and
stability are visible from the permutation action, and a prime near $n$
avoids two selected subgroup orders. Nothing in the argument depends on a
fixed value of the socle multiplicity.

\subsection{Graph fusion: flag parabolics in projective linear groups}

The next proposition exhibits the feature that ordinary examples do not show.
A graph automorphism fuses opposite maximal-parabolic classes, so one replaces
them by parabolics attached to invariant flags. Lemma~\ref{lem:P} then recovers
maximal supplements even though the coordinate subgroups themselves need not
be maximal in $S$.

For $S=\PSL(n,q)$, the \emph{graph part} of $X$ means the image of $X$ in
the quotient of $\Out(S)$ by its diagonal and field subgroup. Thus $X$ has
nontrivial graph part precisely when this image contains the nontrivial
$A_{n-1}$ diagram involution. Appendix~\ref{app:graph} gives the corresponding
definition for the other graph and graph--field families.

\begin{proposition}[Projective flag-parabolic application]
\label{prop:psl-flag}
Let $S=\PSL(n,q)$ with $n\geq4$, and let
$\Inn(S)\leq X\leq\Aut(S)$ have nontrivial graph part. Then the hypotheses of
Theorem~\ref{thm:D} hold for this $X$. Consequently no admissible group with
socle $S^k$, $k\geq2$, and coordinate closure $X$ has property $\mathrm{P}$.
\end{proposition}

\begin{proof}
Write $q=p^f$, $d=\gcd(n,q-1)$, and number the nodes of the $A_{n-1}$
diagram in the usual order. The graph involution $\delta$ sends node $i$ to
$n-i$.

\emph{Selected classes.} If $n=4$, use the maximal parabolic $P_2$ and
the flag parabolic $Q_{\{2\}}$, the stabilizer of an incident
point--hyperplane pair. If $n\geq5$, put
\[
 J_1=\Delta\setminus\{1,n-1\},\qquad
 J_2=\Delta\setminus\{2,n-2\},
\]
and use $Q_{J_1}$ and $Q_{J_2}$, the stabilizers of an incident
point--hyperplane flag and an incident line--$(n-2)$-space flag.

\emph{Maximality and stability.} For $n=4$, node $2$ is fixed by $\delta$,
so $[P_2]$ is $X$-stable. The only proper parabolic overgroups of
$Q_{\{2\}}$ are $P_1$ and $P_3$, which are exchanged by $\delta$; hence
Lemma~\ref{lem:P} applies. For $n\geq5$, both deleted node sets are
$\delta$-orbits. The proper parabolic overgroups of $Q_{J_1}$ are
$P_1,P_{n-1}$, and those of $Q_{J_2}$ are $P_2,P_{n-2}$; each pair is
exchanged by $\delta$. Lemma~\ref{lem:P} applies to both classes.

\emph{Obstruction prime and occurrence in $|S|$.} If $n=4$, let $r$ be a
primitive prime divisor of $p^{4f}-1$; then
$r\mid\Phi_4(q)\mid|S|$. If $n\geq5$, let $r$ be a primitive prime
divisor of $p^{nf}-1$; then $r\mid\Phi_n(q)$, and this cyclotomic
factor occurs in the order of $S$ up to the central divisor $d$.

\emph{Avoidance of the selected subgroups.} For $n=4$, the relevant Levi
exponents are at most $2f<4f$. For $n\geq5$, the Levi blocks have sizes
$(1,n-2,1)$ and $(2,n-4,2)$, so their exponents are at most
$(n-2)f$ and $\max(2,n-4)f$, both smaller than $nf$. Primitivity therefore
gives $r\nmid|V||W|$ for the selected pair.

\emph{Outer contribution.} Since $x=|X/\Inn(S)|$ divides $2df$ and
$r\equiv1\pmod{nf}$ (with $n=4$ in the first branch), we have
$r\nmid2f$; also $r\nmid d$, since $r\mid q-1$ would force
$\ord_r(p)\mid f$. Hence $\vp{r}(x)=0<\vp{r}(|S|)$.

\emph{Exceptions.} There is no positive-base Zsigmondy exception when
$n=4$. For $n\geq5$, the only one is $p^{nf}=2^6$, which gives
$(n,f)=(6,1)$. Take $r=31$; the two flag Levi orders contain only
$\GL_4(2)$- and $\GL_2(2)$-factors and $2$-powers, while
$\vp{31}(|\PSL(6,2)|)=1$ and $x\mid2$.

For this fixed $X$, Theorem~\ref{thm:D} now excludes every admissible group
with socle $S^k$, $k\geq2$, and coordinate closure $X$.
\end{proof}

This application isolates the graph-fusion yield: the stable object is an
invariant flag, not an individual fused maximal-parabolic class. The
stable-poset lemma makes that replacement available uniformly in every socle
multiplicity.

\section{Classification coverage and proof of the main theorem}\label{sec:coverage}

The preceding sections contain the reduction, the reusable criterion, and two
representative applications. Table~\ref{tab:coverage} routes every remaining
simple family to a detailed appendix proof or a finite certificate claim; it
also identifies the common source of the obstruction prime and the exceptional
route. Thus the logical path to Theorem~\ref{thm:main} is visible without
reading the repetitive calculations first.

\begin{table}[htbp]
\scriptsize
\begin{tabular}{@{}L{0.18\textwidth}L{0.23\textwidth}L{0.18\textwidth}L{0.19\textwidth}L{0.13\textwidth}@{}}
\toprule
Family & Stable classes & Prime source & Exceptional handling & Detailed proof \\
\midrule
Alternating & Intransitive and imprimitive maximals & Prime near $n$ & Small degrees certified & \S\ref{sec:applications} \\
$\PSL(2,q)$ & Borel or torus normalizers & Defining or primitive prime & Small fields certified & App.~\ref{app:psl2} \\
Lie type, no graph fusion & Two maximal parabolics & Zsigmondy & Finite base or substitute prime & App.~\ref{app:nograph} \\
Twisted rank at least two & Relative maximal parabolics & Zsigmondy & Listed small parameters & App.~\ref{app:twisted2} \\
Twisted rank one & Borel and geometric maximal & Zsigmondy & Small fields certified & App.~\ref{app:twisted1} \\
Graph-automorphism families & Fixed parabolics, stable flags, split-torus normalizers, or stable fixed-point/subfield classes & Family-specific primitive prime & Finite certificates or substitutes & \S\ref{sec:applications}, App.~\ref{app:graph} \\
Sporadic and Tits & Certified stable maximal classes & Certificate prime & Exact finite records & App.~\ref{app:finite} \\
\bottomrule
\end{tabular}
\caption{Coverage of the finite simple groups by the stable-class criterion.}
\label{tab:coverage}
\end{table}

\begin{proposition}[Coverage]\label{prop:coverage}
Every non-abelian finite simple group is excluded in the sense of
Convention~\ref{conv:excluded}.
\end{proposition}

\begin{proof}
By the classification of finite simple groups
\cite[Ch.~1, Table~I, pp.~8--10]{GLS1} (see also
\cite[Thm.~1.1]{RoneyDougal2021}), $S$ is alternating,
of Lie type, or one of the $26$ sporadic groups. The cases are as follows.
The simplicity exceptions are
those in \cite[Thm.~2.2.7(a)]{GLS3}; the low-rank classical
identifications and the exceptional derived-group identifications are
recorded in \cite[\S2.4, pp.~xi--xii, and \S3.5, p.~xv]{ATLAS}
(with the individual \textsc{Atlas} entries pinpointed in the source map in
\cite{Repo}).
\begin{itemize}
\item $A_n$ ($n\geq5$): Theorem~\ref{thm:an}.
\item $\PSL(2,q)$ ($q\geq4$): Theorem~\ref{thm:psl2}.
\item $\PSL(n,q)$ ($n\geq3$), $D_n(q)$ ($n\geq4$), $E_6(q)$:
Theorem~\ref{thm:graph}.
\item $\PSp(2n,q)$: $n\geq3$, or $n=2$ with $q$ odd:
Theorem~\ref{thm:nograph}; $\Sp(4,2^f)$, $f\geq2$:
Theorem~\ref{thm:graph} ($\Sp(4,2)$ is not simple).
\item $\Omega(2n{+}1,q)$, $q$ odd, $n\geq3$: Theorem~\ref{thm:nograph}
(for $q$ even, $B_n(q)\cong C_n(q)$, already listed;
$\Omega(5,q)\cong\PSp(4,q)$).
\item $E_7(q)$, $E_8(q)$; $F_4(q)$, $p\neq2$; $G_2(q)$, $p\neq3$:
Theorem~\ref{thm:nograph}. $F_4(2^f)$ and $G_2(3^f)$:
Theorem~\ref{thm:graph}. ($G_2(2)$ and ${}^2G_2(3)$ are not simple;
their derived groups $\PSU(3,3)$ and $\PSL(2,8)$ are covered.)
\item $\PSU(n,q)$: $n=3$: Theorem~\ref{thm:twisted1}; $n\geq4$:
Theorem~\ref{thm:twisted2}.
\item ${}^2D_n(q)$ ($n\geq4$), ${}^3D_4(q)$, ${}^2E_6(q)$,
${}^2F_4(2^f)$ ($f$ odd $\geq3$): all four families are covered by
Theorem~\ref{thm:twisted2}; the Tits group ${}^2F_4(2)'$ by
Proposition~\ref{prop:sporadic}. (${}^2D_2$, ${}^2D_3$ are
$\PSL(2,q^2)$, $\PSU(4,q)$, already listed.)
\item $\Sz(2^f)$ ($f$ odd $\geq3$), ${}^2G_2(3^f)$ ($f$ odd $\geq3$):
Theorem~\ref{thm:twisted1}.
\item Sporadic groups: Proposition~\ref{prop:sporadic}.
\end{itemize}
Every finite simple group appears (repetitions from exceptional
isomorphisms are harmless), so the proposition follows.
\end{proof}

\begin{proof}[Proof of Theorem~\ref{thm:main}]
Suppose not, and let $G$ have least order among the non-soluble groups
with property $\mathrm{P}$. By Proposition~\ref{prop:min}, $G$ has a
unique minimal normal subgroup $N=S^k$, $S$ non-abelian simple, and
after conjugation $N\leq G\leq X\wr S_k$ as in Convention~\ref{conv:X}.
If $k=1$, then $G$ is almost simple, contrary to the theorem of
Tikhonenko--Tyutyanov \cite{TT2010}. Hence $k\geq2$. By
Proposition~\ref{prop:coverage} there are classes
$[V]\neq[W]$, maximal in $\Pos_X(S)$ and normally saturating, and a
prime $p$ with $\vp{p}(|S|)-\vp{p}(|V|)-\vp{p}(|W|)>\vp{p}(x)$. By
Theorem~\ref{thm:D}, $G$ does not have property $\mathrm{P}$, a
contradiction.
\end{proof}

\begin{remark}
The argument proves the following stronger statement: for \emph{every}
finite group $G$ with a unique minimal
normal subgroup $S^k$ ($S$ non-abelian simple, $k\geq2$), some pair of
non-conjugate maximal subgroups of $G$ fails to factorize $G$. The
reduction of \S\ref{sec:reduction} is needed only to pass from an
arbitrary non-soluble group to this configuration.
\end{remark}

\section{Conclusion}\label{sec:conclusion}

The transferable result is the stable-class divisibility criterion: two
local classes with a strict valuation gap rule out every multiplicity $k$ at
once. Its subgroup hypotheses are structural rather than cosmetic.
$X$-stability and self-normalization give exact coordinate intersections and
orders, while maximality in $\Pos_X(S)$ and normal saturation are what promote
the resulting normalizers to maximal supplements. Flag parabolics show that
the method can survive automorphism fusion even when no individual maximal
parabolic class is stable.

The global coverage still depends on CFSG and on published maximal-subgroup,
order, automorphism, parabolic, and primitive-prime-divisor results. The finite
and sporadic residue is computationally certified, and the formal developments
verify only the conditional components stated in Appendix~\ref{app:trust};
they do not replace those external inputs or give an end-to-end formal proof.

Li and Yang's subsequent product--socle lifting preprint \cite{LiYang2026}
gives a structurally different route through the direct-power branch. The same
wreath-top budget also appears in the exact-complement problem
\cite{Sater1868}, but its final obstruction is different. Natural next
questions are whether the supplement criterion applies to other prescribed
products of maximal subgroups, whether its stable-poset hypotheses admit a
classification-free source in broader families, and whether the remaining
parabolic and finite-coverage interfaces can be formalized without importing
the full classification ledger into a proof assistant.

\appendix

\section{Supplementary and infinite-family proofs}\label{app:families}

This appendix first records the unused $k=1$ comparison and then discharges the
repetitive infinite-family obligations while the main body retains the two
mechanisms that teach the method. Every family case follows the same audit
order: \emph{selected classes}; \emph{maximality and
stability}; \emph{obstruction prime}; \emph{occurrence in $|S|$};
\emph{avoidance of the selected subgroup orders}; \emph{outer contribution};
\emph{Zsigmondy or small-parameter exceptions}; and \emph{conclusion by the
stable-class criterion}. Closely related branches share a heading when the
same argument proves several of these items at once.

\subsection*{Supplementary almost-simple argument}\label{app:almost-simple}

\begin{proof}[Proof of Theorem~\ref{thm:Dprime}]
\emph{Supplement and order.} By the Frattini argument via $X$-stability,
as in Lemma~\ref{lem:A}, we have $B_VS=G$. Moreover,
$B_V\cap S=N_S(V)=V$, and hence $|B_V|=t|V|$.

\emph{Maximality.} The proof of Lemma~\ref{lem:B} applies without its
coordinate-kernel step; normal saturation enters only through that step and is
not needed at $k=1$ (the hypothesis is kept for uniformity with
Theorem~\ref{thm:D}). Let $B_V\leq H\leq G$ and put $T:=H\cap S$, an
$H$-invariant subgroup containing $V$. If $T=S$, then $H\supseteq S$ and
$H=HS=G$. Otherwise $G=B_VS=HS$. For $g=hs$, with $h\in H$ and
$s\in S$, one has $T^g=(T^h)^s=T^s$; hence $[T]$ is $G$-stable and,
since $X=G$, also $X$-stable. The normalizer tower of $T$ ends at a proper
self-normalizing subgroup $U$. Stability propagates up the tower, so
$[U]\in\Pos_X(S)$ and $[U]\geq[V]$. Poset maximality of $[V]$ forces
$[U]=[V]$, and $V\leq T\leq U$ gives $T=V$. Thus
$H\leq N_G(V)=B_V$.

\emph{Nonconjugacy.} If $B_V^g=B_W$ with $g\in G$, then
$V^g=(B_V\cap S)^g=B_W\cap S=W$, and $X$-stability of $[V]$ gives
$[W]=[V]$, a contradiction.

\emph{Arithmetic.} Property $\mathrm{P}$ would force $G=B_VB_W$, so
$t\,|V||W|/|S|\in\Z$, equivalently $p^\eta\mid t=x$, contrary to
$\eta>\vp{p}(x)$.
\end{proof}

Applying Theorem~\ref{thm:Dprime} to the witnesses in
Proposition~\ref{prop:coverage} gives a supplementary proof that no almost
simple group has property $\mathrm{P}$. This comparison is not used in the
proof of Theorem~\ref{thm:main}, whose $k=1$ branch continues to use
Tikhonenko--Tyutyanov \cite{TT2010}.

\subsection*{Common inputs for the family analysis}\label{app:common-inputs}

Three simplifications organize the ledger. First, when both selected classes
consist of maximal subgroups of $S$, the stable-poset and normal-saturation
hypotheses are automatic, so only stability and the valuation inequality need
checking. Second, in the twisted families the automorphism structure induces
no nontrivial permutation of the relevant relative-diagram types; diagonal
and field automorphisms preserve the chosen parabolic classes. This includes
the triality groups ${}^3D_4(q)$ and the Ree groups. Third, when a graph
automorphism fuses maximal-parabolic classes, invariant flag parabolics become
maximal stable classes by Lemma~\ref{lem:P}; this is the required
flag-parabolic substitute for an individual fused maximal class.

Throughout this section $S$ is simple, $X$ is any group with
$\Inn(S)\leq X\leq\Aut(S)$, and $x=|X/\Inn(S)|$. By
Remark~\ref{rem:maxauto}, when the two exhibited classes consist of
maximal subgroups of $S$, only their $X$-stability and the valuation
inequality require proof. We use Zsigmondy's theorem
\cite[specialized theorem, p.~283]{Zsigmondy}
throughout: for integers $a\geq2$, $e\geq3$ with $(a,e)\neq(2,6)$, there
is a prime $r$ with $\ord_r(a)=e$ (a \emph{primitive prime divisor} of
$a^e-1$); such $r$ satisfies $r\equiv1\pmod e$. For $q=p^f$ we write
$\Phi_e=\Phi_e(q)$ for cyclotomic values and always apply Zsigmondy to
the prime $p$ with exponent $ef$, so $\ord_r(p)=ef$ and
$r\equiv1\pmod{ef}$. The source's additional positive-base exception has
exponent $2$, which lies outside every invocation here and therefore does not
apply. An invocation-by-invocation audit of the hypotheses and exceptions
is part of \cite{Repo}. The $b=1$ specialization and both exceptions are
independently restated in \cite[Thm.~2.2, p.~2]{Jones2007}.

The simple-group and outer-automorphism order formulas used below are
those of \cite[Tables~5--6, p.~xvi]{ATLAS}; for the classical families,
the same formulas and their central divisors are also displayed in
\cite[\S\S2.1--2.4, pp.~x--xii]{ATLAS}. As an independent published
derivation and tabulation we use
\cite[\S10, pp.~220--221, and summary table, p.~239]{Carter1965}.
For parabolics, the statement that the unipotent radical is extended by
the group obtained from the complementary subdiagram (with node-orbits in
the twisted case) is \cite[\S3.6, p.~xv]{ATLAS}; the Levi decomposition
and its root-subgroup description are
\cite[\S8.5, pp.~118--120]{Carter}. Thus the cyclotomic bounds below are
not inferred from numeric tests: they come from the displayed group-order
formulas applied to the named Levi types. The formula-to-source and
parameter audit, including each exceptional substitute, is archived in
\cite{Repo}. In the exceptional case over $\F_2$ below,
$\Omega^+(8,2)$ and $P\Omega^+(8,2)$ denote the same simple group; we use
the shorter former notation when naming that finite group. For classical
groups, argument notation such as $\Sp(4,2^f)$ denotes a named simple group,
whereas subscript notation such as $\Sp_2(q)$ and $\GL_2(q)$ is used for Levi
factors.

\subsection{Projective special linear groups of rank one}\label{app:psl2}

\begin{theorem}\label{thm:psl2}
$S=\PSL(2,q)$, $q=p^f\geq4$, is excluded.
\end{theorem}

\begin{proof}
\emph{Finite exceptions and order data.}
For $q\in\{4,5,7,8,9,11\}$ this is covered by
the finite GAP certificates of Appendix~\ref{app:finite}. All other simple parameters are covered
uniformly below: odd $q\geq13$ by Case~A and even $q=2^f$, $f\geq3$, by
Case~B. Recall
$|S|=q(q^2-1)/d$ with $d=\gcd(2,q-1)$, and
$\Out(S)=C_d\times C_f$, so $x\mid df$: that is, $x\mid2f$ for $q$ odd
and $x\mid f$ for $q$ even.

\emph{Case A---selected classes, maximality, and stability.}
Suppose $q$ is odd and $q\geq13$. Let $U:=N_S(T_+)$ and $V:=N_S(T_-)$
be the normalizers of the split and nonsplit maximal tori. These are dihedral
groups of orders $q-1$ and $q+1$, respectively. Both subgroups are maximal in $S$
for $q\geq13$ by Dickson's theorem \cite{Dickson,Huppert}; see also
\cite[Tables~8.1--8.2, p.~376, and Table~8.7, p.~380]{BHRD}.
All split maximal tori are $S$-conjugate, as are all nonsplit maximal
tori, and every automorphism preserves each type (there is no
graph automorphism in type $A_1$; diagonal and field automorphisms act
algebraically), so $[U]$ and $[V]$ are $X$-stable for every $X$.

\emph{Case A---prime, occurrence, avoidance, and outer contribution.}
Here the obstruction prime is the defining prime $p$ and
$|U||V|/|S|=d/q$, so
$\eta=\vp{p}(|S|)-0-0=f$, while $x\mid2f$ gives
$\vp{p}(x)\leq\vp{p}(f)\leq\log_pf<f=\eta$. Theorem~\ref{thm:D} applies.

\emph{Case B---selected classes, maximality, and stability.}
Let $q=2^f$ with $f\geq3$. Here $|S|=q(q^2-1)$ and $x\mid f$.
Let $U$ be a Borel subgroup (the normalizer of a Sylow $2$-subgroup), of
order $q(q-1)$, and let $V:=N_S(T_+)$ be dihedral of order $2(q-1)$.
Both $U$ and $V$ are maximal for $q\geq8$ \cite{Dickson,Huppert}; see also
\cite[Tables~8.1--8.2, p.~376]{BHRD}. $[U]$ is the class of Sylow-$2$
normalizers, and $[V]$ is the unique class of split torus normalizers;
both are $X$-stable.

\emph{Case B---prime, occurrence, avoidance, outer contribution, and exception.}
Now $|U||V|/|S|=2(q-1)/(q+1)$ with $q+1$ odd and coprime to
$2(q-1)$, so it suffices to find a prime $r\mid q+1$ with
$\vp{r}(q+1)>\vp{r}(f)$. If $f\neq3$, take $r$ a primitive prime divisor
of $2^{2f}-1$: then $r\mid(2^{2f}-1)/(2^f-1)=q+1$ and
$r\equiv1\pmod{2f}$, so $r>f$ and
$\vp{r}(x)=0<\vp{r}(q+1)=\eta$. If $f=3$, then $q=8$ and
$q+1=9$; taking $r=3$ gives $\eta=2>1\geq\vp{3}(x)$.

\emph{Conclusion.} Theorem~\ref{thm:D} applies in both parity branches,
and the finite certificates route the remaining simple parameters.
\end{proof}

\subsection{Lie type without graph symmetry}\label{app:nograph}

\begin{theorem}\label{thm:nograph}
Let $S$ be a simple group of Lie type of rank $\geq2$ over $\F_q$,
$q=p^f$, of a type with no symmetry of its Dynkin diagram in the given
characteristic:
\[
C_n=\PSp(2n,q)\ (n\geq2,\ p\ \text{odd if}\ n=2),\quad
B_n=\Omega(2n{+}1,q)\ (n\geq3,\ q\ \text{odd}),
\]
\[
G_2(q)\ (p\neq3),\quad F_4(q)\ (p\neq2),\quad E_7(q),\quad E_8(q).
\]
Then $S$ is excluded.
\end{theorem}

\begin{proof}
\emph{Torus exponent and selected classes.}
Write $h$ for the relevant torus exponent: $h=2n$ for $B_n/C_n$, and
$h=6,12,18,30$ for $G_2,F_4,E_7,E_8$.

Let $P_1,P_2$ be maximal parabolic subgroups from two distinct classes
(rank $\geq2$ guarantees these; take the ends of the diagram). They are
maximal subgroups of $S$: any overgroup of a parabolic contains a Borel
subgroup, and every subgroup containing a Borel is parabolic
\cite[Thm.~8.3.2, p.~112]{Carter}.

\emph{Maximality and automorphism stability.} $\Aut(S)$ is generated by inner, diagonal, field and
graph automorphisms \cite{Steinberg},
\cite[Def.~3.1 and Thm.~3.4, pp.~31--32]{BrotoMollerOliver2019}. Diagonal
and field automorphisms preserve types of vertices of the building,
hence each parabolic class; type-permuting automorphisms arise only from
diagram symmetries, absent for the listed types in the listed
characteristics (the exceptional symmetries of $B_2/C_2$, $F_4$ in
characteristic~$2$ and $G_2$ in characteristic~$3$ are excluded by
hypothesis and treated in Theorem~\ref{thm:graph}). So every parabolic
class is $X$-stable for every $X$.

\emph{Obstruction prime, occurrence, subgroup avoidance, and outer
contribution.} Let $r$ be a primitive prime divisor of
$p^{hf}-1$. Then:
(i) $r\mid|S|$, since the order formula of each listed type contains the
factor $q^h-1$ (equivalently $\Phi_h(q)$)
\cite[Table~6, p.~xvi]{ATLAS}.
(ii) $r$ divides no maximal parabolic order: $|P_i|=q^{N_i}|L_i|
(q-1)^{c_i}$ with $L_i$ the semisimple part of the Levi factor, whose
order is a product of a $q$-power and factors $q^j-1$ with $j\leq h_i$,
$h_i$ the largest torus exponent of $L_i$; in every listed case every
proper Levi has $h_i<h$. For $B_n$ and $C_n$, the Levi types are respectively
$A_{i-1}\times B_{n-i}$ and $A_{i-1}\times C_{n-i}$; in either case the
relevant exponent is at most $\max(i,2(n-i))<2n$. For the remaining types:
$G_2$ has Levis
$A_1$, $j\leq2$; $F_4$: Levis $B_3,C_3,A_2{\times}A_1$, $j\leq6$; $E_7$:
Levis $E_6,D_6,A_6,\dots$, exponents $\leq12$; $E_8$: Levis
$E_7,D_7,A_7$, exponents $\leq18$. By primitivity $r$ divides none of
these, nor $q$, nor $q-1$.
(iii) $r\nmid x$: $r\equiv1\pmod{hf}$ with $h\geq4$, so
$r>hf\geq4f\geq d_S f\geq x$, where $d_S\leq\gcd(2,q-1)^{\,2}\leq4$
bounds the diagonal part and there is no graph part.

Thus $\eta=\vp{r}(|S|)\geq1>0=\vp{r}(x)$ for the pair $([P_1],[P_2])$
whenever $r$ exists.

\emph{Zsigmondy and small-parameter exceptions.} $p^{hf}=2^6$ forces $p=2$ and $hf=6$:
either $G_2(2)$ (not simple; $G_2(2)'\cong\PSU(3,3)$ is in the finite
base) or $\Sp(6,2)$ (finite base). For $h\in\{12,18,30\}$, $hf=6$ is
impossible; $B_n$ has $q$ odd.

\emph{Conclusion.} Theorem~\ref{thm:D} applies to every nonexceptional
parameter, and Appendix~\ref{app:finite} supplies the two routed cases.
\end{proof}

\subsection{Twisted groups of twisted rank at least two}\label{app:twisted2}

\begin{theorem}\label{thm:twisted2}
Let $S$ be one of
\[
\begin{gathered}
\PSU(n,q)\ (n\geq4),\quad
P\Omega^-(2n,q)={}^2D_n(q)\ (n\geq4),\\
{}^3D_4(q),\quad {}^2E_6(q),\quad
{}^2F_4(q)\ (q=2^f,\ f\ \text{odd}\geq3).
\end{gathered}
\]
Then $S$ is excluded.
\end{theorem}

\begin{proof}
\emph{Selected classes and maximality.} Each of these groups has a twisted $BN$-pair of rank at least $2$, and
therefore has at least two classes of maximal parabolic subgroups $P_1,P_2$.
These subgroups are maximal in $S$, by the same argument as
in Theorem~\ref{thm:nograph}: every finite group of Lie type in defining
characteristic has a split $BN$-pair
\cite[Example~1.16 and \S1.3.5, pp.~10--11]{Hiss2011}, and Carter's
overgroup and self-normalizer theorems are stated for an arbitrary group
with a $BN$-pair \cite[Thms.~8.3.2--8.3.3, p.~112]{Carter}.

\emph{Automorphism stability.}
\emph{Parabolic preservation.} Parabolic subgroups are precisely the
overgroups of Borel subgroups, which are the normalizers of Sylow
$p$-subgroups in defining characteristic
\cite[Example~1.16 and \S1.3.5, pp.~10--11]{Hiss2011}. This description is
automorphism-invariant. Hence every automorphism permutes the parabolic
classes and acts on the relative type set.

\emph{Relative-type preservation.} For a twisted group, the standard
automorphism decomposition has no separate graph factor
\cite[Def.~3.3 and Thm.~3.4, p.~32]{BrotoMollerOliver2019}. The
inner-diagonal and Frobenius representatives centralizing the defining
Steinberg endomorphism preserve each orbit of simple roots, hence each node of
the relative diagram. They normalize the standard Borel and torus, so they
map every standard parabolic $Q_J$ to a conjugate of the same relative type.
Inner automorphisms fix classes. Consequently every parabolic class is
$X$-stable for every $X$.

\emph{Rank-two diagnostic.} In ${}^2F_4(q)$ the two maximal parabolic
classes also cannot be interchanged because their Levi factors
${}^2B_2(q)$ and $\SL_2(q)$ are non-isomorphic
\cite[Thm.~4.5(i)--(ii), p.~166]{Wilson2009}. Diagram asymmetry alone would
not suffice: the relative diagrams of $\PSU(4,q)$ and $\PSU(5,q)$ have a
symmetric underlying graph, and the untwisted group $\Sp(4,2^f)$ admits a
type-swapping automorphism.

\emph{Obstruction prime, occurrence in $|S|$, and subgroup avoidance.}
For each type, choose the exponent $e$ and
let $r$ be a primitive prime divisor of $p^{e}-1$. In each case $r$ divides
$|S|$ but neither $|P_1|$ nor $|P_2|$:
\begin{itemize}
\item $\PSU(n,q)$, $n$ odd: $e=2nf$, so $r\mid\Phi_{2n}(q)\mid q^n+1$,
which divides $|\SU(n,q)|=q^{n(n-1)/2}\prod_{i=2}^n(q^i-(-1)^i)$. The
Levi factor of $P_i$ ($i=1,2$) is a central product of $\GL_i(q^2)$ and
$\GU(n-2i,q)$: its $p'$-order is a product of factors $q^{2j}-1$
($j\leq2$) and $q^j-(-1)^j$ ($j\leq n-2$), all dividing $p^m-1$ with
$m\leq\max(4f,2(n-2)f)<2nf$.
These order and Levi formulas are recorded in
\cite[\S2.2, p.~x, and \S3.6, p.~xv]{ATLAS}.
\item $\PSU(n,q)$, $n$ even: $e=2(n-1)f$,
$r\mid\Phi_{2(n-1)}(q)\mid q^{n-1}+1$ (the $i=n-1$ factor). Levi
exponents are $\leq\max(4f,2(n-3)f,(n-2)f)<2(n-1)f$ for $n\geq4$
\cite[\S2.2, p.~x, and \S3.6, p.~xv]{ATLAS}.
\item ${}^2D_n(q)$, $n\geq4$: $e=2nf$, $r\mid\Phi_{2n}(q)\mid q^n+1$,
which divides $q^{n(n-1)}(q^n+1)\prod_{i=1}^{n-1}(q^{2i}-1)$.
The Levi factor of $P_i$ has type
$\GL_i(q)\times O^-(2(n-i),q)$, and its relevant exponents are at most
$\max(2f,2(n-1)f)<2nf$
\cite[\S2.4, p.~xii, and \S3.6, p.~xv]{ATLAS}.
\item ${}^3D_4(q)$: $e=12f$, $r\mid\Phi_{12}(q)=q^4-q^2+1$, which
divides $q^8+q^4+1=(q^4+q^2+1)(q^4-q^2+1)$. Since
$|{}^3D_4(q)|=q^{12}(q^8+q^4+1)(q^6-1)(q^2-1)$, we have
$r\mid|S|$. The two Levi factors
are $\SL_2(q^3)\cdot(q-1)$ and $\SL_2(q)\cdot(q^3-1)$ (mod centers):
$p'$-orders divide $(q^6-1)(q-1)$ and $(q^2-1)(q^3-1)$, exponents
$\leq6f<12f$ \cite[Table~6, p.~xvi, and \S3.6, p.~xv]{ATLAS}.
\item ${}^2E_6(q)$: $e=18f$, $r\mid\Phi_{18}(q)\mid q^9+1$, which
divides $q^{36}(q^{12}-1)(q^9+1)(q^8-1)(q^6-1)(q^5+1)(q^2-1)$. The four
maximal parabolic classes correspond to the four $\tau$-orbits of nodes
of $E_6$ ($\tau$ the order-two symmetry): $\{2\}$, $\{4\}$, $\{1,6\}$,
$\{3,5\}$; the complementary $\tau$-stable subdiagrams are $A_5$,
$A_2{\times}A_1{\times}A_2$, $D_4$, $A_2{\times}A_1{\times}A_1$ with induced
twists, so the Levi derived groups are $\SU(6,q)$,
$\SL_2(q){\times}\SL_3(q^2)$, ${}^2D_4(q)$-type, and
$\SL_3(q){\times}\SL_2(q^2)$, with cyclotomic exponents
$\leq10f,\,6f,\,8f,\,4f$ respectively, all $<18f$. Any two classes
serve as $P_1,P_2$ \cite[Table~6, p.~xvi, and \S3.6, p.~xv]{ATLAS}.
This diagram derivation agrees with the published list
$({}^2A_5(q),{}^2D_4(q),A_1(q){\times}A_2(q^2),
A_1(q^2){\times}A_2(q))$ in
\cite[Lemma~6.8, p.~82]{Montinaro2024}.
\item ${}^2F_4(q)$, $f$ odd $\geq3$: $e=12f$,
$r\mid\Phi_{12}(q)\mid q^6+1$, which divides
$|{}^2F_4(q)|=q^{12}(q^6+1)(q^4-1)(q^3+1)(q-1)$. The two Levi derived
groups are ${}^2B_2(q)\cong\Sz(q)$ and $\SL_2(q)$
\cite[Thm.~4.5(i)--(ii), p.~166]{Wilson2009} and
\cite[Main Theorem, pp.~52--53]{Malle1991}: $p'$-orders divide
$(q^2+1)(q-1)^2$ and
$(q^2-1)(q-1)$, exponents $\leq4f<12f$. (${}^2F_4(2)$ is not simple;
its derived group, the Tits group, is in the finite base.)
\end{itemize}

\emph{Outer contribution.} In every case $x$ divides $d_S\cdot c f$, where the
diagonal order $d_S$ divides $q+1$ (types ${}^2A$, ${}^2E_6$) or $4$
(type ${}^2D$) or is $1$ (${}^3D_4$, ${}^2F_4$), and the field part has
order $cf$ with $c\leq3$
\cite[Table~5, p.~xvi]{ATLAS} and
\cite[Def.~3.3 and Thm.~3.4, p.~32]{BrotoMollerOliver2019}.
Since $r\equiv1\pmod e$ with $e\geq6f$, $r$
exceeds the field-part order. If $r\mid d_S$, then either $r\mid q+1$,
which would give $\ord_r(p)\mid2f<e$, or $r\leq4$, contradicting
$r>e\geq6$. Hence
$\eta=\vp{r}(|S|)\geq1>0=\vp{r}(x)$ whenever the primitive prime exists.

\emph{Zsigmondy and small-parameter exceptions.} $r$ fails to exist only if $p^e=2^6$. Here
$e\geq6$ always, and $e=6$ occurs only for $\PSU(4,q)$ with $f=1$; then
$p^6=2^6$ forces $S=\PSU(4,2)\cong\PSp(4,3)$
\cite[Thm.~5.3, p.~61]{CameronNotes}; this group, of order $25920$, is
settled by the finite certificate of Appendix~\ref{app:finite}. No other case
reaches $p^e=2^6$.

\emph{Conclusion.} Theorem~\ref{thm:D} applies to every uniform branch,
and the sole exceptional simple group is routed to the exact finite record.
\end{proof}

\subsection{Twisted groups of twisted rank one}\label{app:twisted1}

\begin{theorem}\label{thm:twisted1}
Let $S$ be one of
\[
\PSU(3,q)\ (q\geq3),\quad
\Sz(q)={}^2B_2(q)\ (q=2^f,\ f\ \text{odd}\geq3),\quad
{}^2G_2(q)\ (q=3^f,\ f\ \text{odd}\geq3).
\]
Then $S$ is excluded.
\end{theorem}

\begin{proof}
\emph{Common selected class and stability.} Each group has a single class of parabolic subgroups, namely the Borel
subgroups $B=N_S(Q)$, $Q\in\Syl_p(S)$; $B$ is a \emph{maximal}
subgroup (rank-one $BN$-pair) and its class is $X$-stable for every $X$.
We pair it with a second $X$-stable maximal class avoiding a Zsigmondy
prime.

\emph{$\PSU(3,q)$---second class, maximality, and stability.}
For $q\geq8$, $|S|=q^3(q^3+1)(q^2-1)/d$,
$d=\gcd(3,q+1)$; $|B|=q^3(q^2-1)/d$. Let $V$ be the stabilizer of a
nonisotropic point of the unitary geometry. It has order
$q(q-1)(q+1)^2/d$ and is maximal for $q\geq3$, $q\neq5$
\cite[Table~8.5, p.~379]{BHRD} (class $\mathcal{C}_1$). These stabilizers
form a single class, preserved by every semilinear automorphism.

\emph{$\PSU(3,q)$---prime, occurrence, avoidance, and outer contribution.}
Take $r$ a primitive prime divisor of $p^{6f}-1$. It exists because
the exceptional equation $p^{6f}=2^6$ would force $q=2$, for which $S$
is not simple. Then $r\mid\Phi_6(q)=q^2-q+1$,
which divides $(q^3+1)/(q+1)$, and hence $r\mid|S|$;
$\ord_r(p)=6f$ rules out
$r$ dividing $q(q^2-1)$ or $q(q-1)(q+1)^2$. Also $x\mid2df$, and
$r\equiv1\pmod{6f}$ gives $r>2f$, while $r\mid d\mid q+1$ would force
$\ord_r(p)\mid2f$. Thus $\eta:=\vp{r}(|S|)\geq1>0=\vp{r}(x)$.
\emph{$\PSU(3,q)$---small parameters and conclusion.}
For $q\in\{3,4,5,7\}$ (in particular the maximality exception $q=5$),
$S$ is in the finite base. Theorem~\ref{thm:D} handles every other parameter.

\emph{$\Sz(q)$---selected class, maximality, and stability.} Here $|S|=q^2(q^2+1)(q-1)$; $|B|=q^2(q-1)$. Let
$V:=D_{2(q-1)}$ be the normalizer of a cyclic subgroup of order $q-1$.
It is maximal by Suzuki's classification \cite{Suzuki1962}; see also
\cite[Thm.~7.3.5, p.~367, and Table~8.16, p.~384]{BHRD}. These
normalizers form a single class, stable under $\Aut(S)=S\rtimes C_f$.

\emph{$\Sz(q)$---prime, occurrence, avoidance, outer contribution, and
exceptions.} Let $r$ be a primitive prime divisor of $2^{4f}-1$. Such a prime exists
because $4f\geq12$. Then $r\mid\Phi_4(q)=q^2+1$, so $r\mid|S|$,
$r\nmid|B|$, $r\nmid|V|$; $x\mid f$, and
$r\equiv1\pmod{4f}$ gives $r>f$. Thus
$\vp{r}(|S|)\geq1>0=\vp{r}(x)$, and Theorem~\ref{thm:D} applies. The
minimum exponent is $4f\geq12$, so there is no Zsigmondy exception.

\emph{${}^2G_2(q)$---selected class, maximality, and stability.} Here $|S|=q^3(q^3+1)(q-1)$; $|B|=q^3(q-1)$. Let
$V:=C_S(\iota)\cong\langle\iota\rangle\times\PSL(2,q)$ for an involution
$\iota$. The involutions form a single $S$-class, so $[V]$ is
$\Aut(S)$-stable. Moreover, $|V|=q(q^2-1)$, and $V$ is maximal
\cite[Thm.~C, pp.~33--34]{Kleidman1988}.

\emph{${}^2G_2(q)$---prime, occurrence, avoidance, outer contribution, and
exceptions.} Let $r$ be a primitive prime divisor of $3^{6f}-1$. Such a prime exists because $6f\geq18$. Then
$r\mid\Phi_6(q)=q^2-q+1$, which divides $(q^3+1)/(q+1)$, and hence
$r\mid|S|$;
$\ord_r(3)=6f$ rules out $r\mid q^3(q-1)$ and $r\mid q(q^2-1)$;
$x\mid f<r$. So $\vp{r}(|S|)\geq1>0=\vp{r}(x)$.
(${}^2G_2(3)$ is not simple; ${}^2G_2(3)'\cong\PSL(2,8)$ is covered by
Theorem~\ref{thm:psl2}.) The exponent $6f\geq18$ has no Zsigmondy
exception, so Theorem~\ref{thm:D} applies.

The three branches therefore satisfy the common eight-step template and
exclude every twisted-rank-one simple group in the theorem.
\end{proof}

\subsection{Untwisted types with graph symmetry}\label{app:graph}

\begin{theorem}\label{thm:graph}
Let $S$ be one of
\[
\PSL(n,q)\ (n\geq3),\quad P\Omega^+(2n,q)=D_n(q)\ (n\geq4),\quad
E_6(q),
\]
\[
\Sp(4,q)\ (q=2^f,\ f\geq2),\quad F_4(q)\ (q=2^f),\quad
G_2(q)\ (q=3^f).
\]
Then $S$ is excluded.
\end{theorem}

\begin{proof}
We keep the standard-parabolic notation $Q_J$ of Lemma~\ref{lem:P}. The
\emph{graph part} of $X$ is the image of $X$ in the quotient of
$\Out(S)$ by its inner-diagonal and field part: a subgroup of
$C_2=\langle\delta\rangle$ for $\PSL(n,q)$ ($n\geq3$), $D_n$ ($n\geq5$),
$E_6$; a subgroup of $S_3$ (triality) for $D_4$; while for $\Sp(4,2^f)$,
$F_4(2^f)$, $G_2(3^f)$, $\Out(S)$ is cyclic of order $2f$ generated by
the exceptional graph-field automorphism $\rho$ with $\rho^2=\varphi_p$
\cite[Def.~3.1(b), p.~31, and Thm.~3.4, p.~32]{BrotoMollerOliver2019},
and ``$X$ has graph part'' means the image of
$X$ in $\Out(S)$ is not contained in $\langle\rho^2\rangle$. An
automorphism acts on the building permuting vertex types by a diagram
symmetry; automorphisms with trivial graph part preserve every type. By
Convention~\ref{conv:X} we may normalize the graph part up to
$\Aut(S)$-conjugacy; for $D_4$ this lets us assume a $C_2$ graph part is
generated by the symmetry fixing nodes $1,2$ and swapping $3,4$.

The two main branches below follow the appendix template. In Case~A, each
bullet states the selected maximal parabolics and primitive prime, then checks
occurrence, Levi avoidance, the outer bound, and any exception. In Case~B,
Table~\ref{tab:graph-summary} presents those data uniformly; prose is retained
for the exceptional mechanisms that require more than a table entry.

\emph{Case A---trivial graph part: selected classes and arithmetic.}
Every parabolic class is
$X$-stable, and the argument of Theorem~\ref{thm:nograph} applies. The selected
pairs and obstruction primes are given below; the $E_6(q)$ entry refers forward
to a construction valid for every $X$.
\begin{itemize}
\item \emph{$\PSL(n,q)$: pair $([P_1],[P_2])$, $r$ a primitive prime divisor
of $p^{nf}-1$.} Then $r\mid\Phi_n(q)\mid q^n-1$, and this cyclotomic
factor divides $|S|\cdot d$, where $d=\gcd(n,q-1)$. The Levi factors,
of $\GL_{n-1}(q)$-type and $\GL_2{\times}\GL_{n-2}$-type, have
cyclotomic exponents at most $(n-1)f<nf$. Moreover, $x\mid2df$. The
congruence $r\equiv1\pmod{nf}$ rules out $r\mid2f$, since
$r>nf\geq3f>2f$, and also rules out $r\mid d$, since that would imply
$r\mid q-1$ and hence $\ord_r(p)\mid f$. The Zsigmondy exceptions occur
when $p^{nf}=2^6$. If $(n,f)=(3,2)$, then $S=\PSL(3,4)$, which is in the
finite base for all ten $X$. If $(n,f)=(6,1)$, then $S=\PSL(6,2)$;
take instead $r=31$ ($\ord_{31}(2)=5$) and the pair
$([P_2],[P_3])$. Their Levi types $\GL_2{\times}\GL_4$ and
$\GL_3{\times}\GL_3$ have exponents at most $4<5$, so
$31\nmid|P_2||P_3|$, whereas $\vp{31}(|\PSL(6,2)|)=1$ and $x\mid2$.

\item \emph{$D_n(q)$, $n\geq4$: pair $([P_1],[P_2])$ (singular-point and
singular-line stabilizers).} Let $r$ be a primitive prime divisor of
$p^{2(n-1)f}-1$. Then $r\mid\Phi_{2n-2}(q)$, which divides
$q^{n(n-1)}(q^n-1)\prod_{i=1}^{n-1}(q^{2i}-1)$. The Levi factors, of
$D_{n-1}$-type and $A_1{\times}D_{n-2}$-type, have exponents at most
$\max((n-1)f,2(n-2)f)<2(n-1)f$. Moreover, $x\mid24f$, accounting for
diagonal, field, and graph factors of orders at most $4$, $f$, and $6$,
respectively. Since $r\equiv1\pmod{2(n-1)f}$ and $2(n-1)\geq6$, we
have $r>6f$ and $r\nmid24$. The only Zsigmondy exception is
$p^{2(n-1)f}=2^6$, which gives $(n,f)=(4,1)$ and
$S=\Omega^+(8,2)$; Case~B3 handles this group for every $X$.

\item \emph{$E_6(q)$.} The pair used in Case~B4 is stable for every $X$.

\item \emph{$\Sp(4,2^f)$.} Take the pair $([P_1],[P_2])$, and let $r$
be a primitive prime divisor of $2^{4f}-1$. Such a prime exists because
$4f\geq8$. It divides
$\Phi_4(q)=q^2+1$. The Levi factors $\Sp_2(q)$ and $\GL_2(q)$ have
exponents at most $2f<4f$; here $x\mid f$ and $r>4f$.

\item \emph{$F_4(2^f)$.} Take the pair $([P_1],[P_2])$, and let $r$ be a
primitive prime divisor of $2^{12f}-1$. It divides $\Phi_{12}(q)$. All
proper Levi factors, of types
$B_3$, $C_3$, $A_2{\times}A_1$, and $A_1{\times}A_2$, have exponents
at most $6f<12f$, and $x\mid f<r$. This completes the $p=2$ case
excluded from Theorem~\ref{thm:nograph}.

\item \emph{$G_2(3^f)$.} Take the pair $([P_1],[P_2])$, and let $r$ be a
primitive prime divisor of $3^{6f}-1$. It divides $\Phi_6(q)$. Levi factors
of type $A_1$ have
exponents at most $2f<6f$, and $x\mid f<r$. This is the $p=3$ case of
$G_2$.
\end{itemize}

\emph{Case B---nontrivial graph part: stable-class mechanism.}
The flag-parabolic substitutes below are standard parabolics $Q_J$ handled by
Lemma~\ref{lem:P}. Their node sets are invariant under the relevant diagram
symmetry, while their proper parabolic overgroups occur in nontrivial graph
orbits. The lemma therefore supplies stable-poset maximality and normal
saturation. Table~\ref{tab:graph-summary} gives the complete branch roadmap;
the prose following it is reserved for triality, invariant flags, and the
exceptional graph--field fixed-point constructions.

\begingroup
\footnotesize
\setlength{\tabcolsep}{3pt}
\renewcommand{\arraystretch}{1.13}
\begin{longtable}{@{}L{0.105\textwidth}L{0.295\textwidth}L{0.265\textwidth}L{0.255\textwidth}@{}}
\caption{Nontrivial-graph branches: selected classes, arithmetic, and exception routing.}
\label{tab:graph-summary}\\
\toprule
Branch & Stable classes and maximality mechanism & Obstruction prime and Levi avoidance & Outer contribution and exceptions \\
\midrule
\endfirsthead
\multicolumn{4}{c}{\tablename~\thetable\ (continued)}\\
\toprule
Branch & Stable classes and maximality mechanism & Obstruction prime and Levi avoidance & Outer contribution and exceptions \\
\midrule
\endhead
\midrule
\multicolumn{4}{r}{Continued on next page}\\
\endfoot
\bottomrule
\endlastfoot
B1: $\PSL(3,q)$
& The Borel $B=Q_{\varnothing}$ is a flag-parabolic substitute because $P_1,P_2$ are interchanged. The split-torus normalizer
$N_T\cong((q-1)^2/d).S_3$, $d=\gcd(3,q-1)$, is maximal for $q\geq5$ and its class is automorphism-stable
\cite{Mitchell1911,Hartley1925}; see also
\cite[Table~8.3, p.~378]{BHRD}.
& Let $r$ be a primitive prime divisor of $p^{3f}-1$. Then
$r\mid\Phi_3(q)$, $r\nmid|B|$, and
$r\nmid|N_T|=6(q-1)^2/d$ because $r\equiv1\pmod3$ gives $r\geq7$.
& $x\mid2df$, while $r>3f$ and $r\nmid d$. The cases $q\leq4$ are in the finite range; the $2^6$ exception is $\PSL(3,4)$.
\\[2pt]
B2: $\PSL(n,q)$, $n\geq4$
& For $n=4$, use $[P_2]$ and $[Q_{\{2\}}]$; for $n\geq5$, use the two invariant flag parabolics deleting $\{1,n-1\}$ and $\{2,n-2\}$. Proposition~\ref{prop:psl-flag} proves stability and maximal-supplement status.
& Use a primitive divisor of $p^{4f}-1$ when $n=4$, and of $p^{nf}-1$ otherwise. The Levi block sizes are at most $2$, $(n-2)$, and $\max(2,n-4)$, all below the primitive exponent.
& $x\mid2df$ and the primitive prime avoids $2df$. The sole exception is $\PSL(6,2)$, where $r=31$ and the same flag pair applies.
\\[2pt]
B3: $D_n(q)$
& For $n\geq5$, $[P_1],[P_2]$ are fixed by the graph involution. For $D_4$, the same pair handles trivial or normalized $C_2$ graph part; triality uses $[P_2]$ and $[Q_{\{2\}}]$. The triality orbit argument is given below.
& Use a primitive divisor of $p^{2(n-1)f}-1$; for $D_4$ the exponent is $6f$. The relevant Levi exponents are strictly smaller; triality gives $A_1^3$ and $A_1$.
& $x\mid24f$ and the primitive prime avoids this factor. The $2^6$ exception $\Omega^+(8,2)$ uses $r=5$; the two graph branches are checked below.
\\[2pt]
B4: $E_6(q)$
& The diagram involution fixes nodes $2,4$, so $[P_2],[P_4]$ are stable maximal classes for every $X$.
& Let $r$ be a primitive prime divisor of $p^{12f}-1$. Then $r\mid\Phi_{12}(q)$ occurs in $|S|$. Here the Levi of $P_2$ has type $A_5$ with exponents at most $6f$, and the $P_4$ Levi has type $A_2{\times}A_1{\times}A_2$ with exponents at most $3f$ \cite[\S8.5, pp.~118--120]{Carter}.
& With $d=\gcd(3,q-1)$, $x\mid2df\leq6f<r$. There is no Zsigmondy exception.
\\[2pt]
B5: $\Sp(4,2^f)$
& Use the Borel flag-parabolic substitute and $V=\Fix(\widetilde\rho^{\,f})\cap S$: $V=\Sz(q)$ for odd $f$, and $V=\Sp(4,2^{f/2})$ for even $f$. Maximality and stability are proved below.
& For odd $f$, use a primitive divisor of $2^{2f}-1$; for even $f$, one of $2^{4f}-1$. The fixed-point subgroup and Borel avoid it.
& $x\mid2f$. The nonsimple $f=1$ case routes to $A_6$; $q=8$ uses the substitute prime $3$. The finite $q=4$ record checks the same construction.
\\[2pt]
B6: $F_4(2^f)$
& Use the invariant flag parabolics $Q_{\{2,3\}}$ and $Q_{\{1,4\}}$; each pair of proper parabolic overgroups is interchanged by $\rho$. Details appear below.
& A primitive prime divisor of $2^{12f}-1$ avoids both selected classes. Their Levi types are $B_2$ and $A_1{\times}A_1$, with exponents at most $4f$ and $2f$ \cite[\S8.5, pp.~118--120]{Carter}.
& $x\mid2f<12f<r$. There are no exceptions.
\\[2pt]
B7: $G_2(3^f)$
& Use the Borel flag-parabolic substitute and $V=\Fix(\widetilde\rho^{\,f})\cap S$: $V={}^2G_2(q)$ for odd $f$, and $V=G_2(3^{f/2})$ for even $f$. Maximality and stability are proved below.
& For odd $f$, use a primitive divisor of $3^{3f}-1$; for even $f$, one of $3^{6f}-1$. The factors $q^6-1$ and $q^2-1$ in the group order supply the required primes.
& $x\mid2f$. There is no Zsigmondy exception; $G_2(3)$ also has a finite certificate for both coordinate closures.
\\
\end{longtable}
\endgroup

The table is a roadmap rather than a replacement for the exceptional
mechanisms. We now justify the three points that are not immediate from the
standard maximal-parabolic calculation.

\emph{Triality detail (B3).}
For $n\geq5$, the graph part is at most the $C_2$ swapping the two spinor
nodes $n{-}1,n$ and fixing $1,\dots,n{-}2$, so the Case~A pair
$([P_1],[P_2])$ is stable for every $X$. For $n=4$, the same statement
holds for trivial or normalized $C_2$ graph part. If the graph part contains
triality, use $([P_2],[Q_{\{2\}}])$. Node $2$ is fixed by all of $S_3$, so
$[P_2]$ is a stable maximal class with Levi type
$A_1{\times}A_1{\times}A_1$. The flag parabolic $Q_{\{2\}}$ has Levi type
$A_1$; its proper overgroups
$Q_{\{1,2\}},Q_{\{2,3\}},Q_{\{2,4\}}$ and $P_4,P_3,P_1$ form two regular
triality orbits on the end nodes $\{1,3,4\}$. None is $X$-stable, so
Lemma~\ref{lem:P} applies.

Let $r$ be a primitive prime divisor of $p^{6f}-1$. It avoids both classes,
$\vp{r}(|S|)\geq1$, and $r\nmid x$ as in Case~A. If $p=2$ and $f=1$, then
$S=\Omega^+(8,2)$ and $x\mid24$. Take $r=5=\Phi_4(2)$. Then
$\vp{5}(|S|)=\vp{5}\bigl(2^{12}(2^6-1)(2^4-1)^2(2^2-1)\bigr)=2$
and $\vp{5}(x)=0$.
For triality, both selected Levi orders are $\{2,3\}$-numbers, so the gap is
$2$. For trivial or $C_2$ graph part, $\vp{5}(|P_1|)=1$ and
$\vp{5}(|P_2|)=0$, so the gap is $1$.

\emph{Exceptional graph--field fixed points (B5).}
Let $S=\Sp(4,2^f)$ with $f\geq2$ and let $B$ be a Borel subgroup. Its two
proper maximal-parabolic overgroups are interchanged by $\rho$, so
Lemma~\ref{lem:P} makes $[B]$ a stable-poset maximal class. Put
\[
 V:=\Fix(\widetilde\rho^{\,f})\cap S,
\]
where $\widetilde\rho^2=\varphi_2$ and
$\widetilde\rho^{\,2f}=F_q$. Thus $V=\Sz(q)$ for odd $f$, whereas for even
$f$ it is the subfield subgroup $\Sp(4,2^{f/2})$. Both are maximal
\cite[Table~8.14, p.~384]{BHRD} and \cite{Suzuki1962}. Since
$\widetilde\rho$ commutes with its own power, $\rho(V)=V$; inner
automorphisms preserve $[V]$ by conjugation. Hence $[V]$ is
$\Aut(S)$-stable.

\emph{Case $f$ odd, $2f\neq6$.} Let $r$ be a primitive prime divisor of $2^{2f}-1$.
It divides $q+1$, so
$\vp{r}(|S|)=2\vp{r}(q+1)\geq2$, while
$\vp{r}(|B|)=\vp{r}(|\Sz(q)|)=0$ and $x\mid2f<r$. If $2f=6$, take
$r=3$; then $\vp{3}(|\Sp(4,8)|)=\vp{3}(63\cdot4095)=4>1\geq\vp{3}(x)$,
while both selected subgroups have zero $3$-valuation.

\emph{Case $f$ even.} Let $r$ be a primitive prime divisor of $2^{4f}-1$.
Writing $q_0^2=q$, one has
$|\Sp(4,q_0)|=q_0^4(q_0^2-1)(q_0^4-1)$, whose cyclotomic exponents are at
most $2f<4f$. Hence $r$ avoids both $V$ and $B$, while it divides $|S|$
and $x\mid2f<r$. The finite $\Sp(4,4)$ certificate records this same
construction with the primes $5$ and $17$.

\emph{Invariant flag parabolics (B6).}
For $S=F_4(2^f)$, the exceptional graph--field automorphism reverses
$1\!-\!2\Rightarrow3\!-\!4$. The node sets $J_1=\{2,3\}$ and
$J_2=\{1,4\}$ are invariant. The proper overgroups of $Q_{J_1}$ are
$P_4$ and $P_1$, which are interchanged; those of $Q_{J_2}$ are
$Q_{\{1,2,4\}}$ and $Q_{\{1,3,4\}}$, also interchanged and of distinct
types. Lemma~\ref{lem:P} therefore applies to both. The corresponding Levi types are $B_2$ and $A_1{\times}A_1$,
with exponents at most $4f$ and $2f$.
Let $r$ be a primitive prime divisor of $2^{12f}-1$. It avoids both
classes, $\vp{r}(|S|)\geq1$, and $x\mid2f<12f<r$; there is no exception.

\emph{Exceptional graph--field fixed points (B7).}
Let $S=G_2(3^f)$. The automorphism $\rho$ interchanges the two maximal
parabolics and satisfies $\rho^2=\varphi_3$, so the Borel class is a
flag-parabolic substitute. With
$V:=\Fix(\widetilde\rho^{\,f})\cap S$, one has $V={}^2G_2(q)$ for odd $f$
and $V=G_2(3^{f/2})$ for even $f$. These subgroups are maximal
\cite[Thm.~A, p.~33]{Kleidman1988}; for $f=1$, the identification
${}^2G_2(3)\cong\mathrm{P}\Gamma\mathrm{L}_2(8)$ and its maximality in
$G_2(3)$ are also recorded in the \textsc{Atlas} \cite{ATLAS}. The same
commuting-endomorphism argument as in B5 makes $[V]$ automorphism-stable.

\emph{Case $f$ odd.} Let $r$ be a primitive prime divisor of $3^{3f}-1$.
It divides $\Phi_3(q)=q^2+q+1$, and $|S|=q^6(q^6-1)(q^2-1)$ shows that
$r\mid|S|$. Primitivity gives $r\nmid|B|=q^6(q-1)^2$ and
$r\nmid|{}^2G_2(q)|=q^3(q^3+1)(q-1)$; moreover
$r\equiv1\pmod{3f}$ gives $r>2f\geq x$. There is no Zsigmondy exception
because the base is $3$.

\emph{Case $f$ even.} Let $r$ be a primitive prime divisor of $3^{6f}-1$.
Then $r\mid\Phi_6(q)=q^2-q+1$, while
$r\nmid|B|$ and
$r\nmid|G_2(3^{f/2})|=q^3(q^3-1)(q-1)$; also $x\mid2f<6f<r$.
The group $G_2(3)$ additionally carries a finite certificate for both $X$.

\emph{Conclusion.} In every case the selected classes and prime satisfy
Theorem~\ref{thm:D}, which excludes $S^k$ for every $k\geq2$.
\end{proof}

\section{Finite and sporadic certificate claims}\label{app:finite}

This appendix states exactly what the finite computations contribute. Command
lines, package installation, certificate filenames, hashes, and mutation-test
mechanics are kept in the repository supplement \cite{Repo}; none of those
operational details substitutes for the mathematical claims below.

\subsection{Finite-range coverage}

\begin{proposition}[Finite-range coverage]\label{prop:base}
Every non-abelian simple group $S$ with $|S|\leq1.05\times10^7$ is excluded
in the sense of Convention~\ref{conv:excluded}.
\end{proposition}

\begin{proof}
\emph{Inventory and routing.}
The certified inventory is the disjoint union of the $47$ non-abelian simple
groups with $|S|<5\times10^5$ and the $51$ groups with
$5\times10^5\leq|S|\leq1.05\times10^7$. The committed replay regenerates the
upper range with GAP's simple-group iterator. The lower range is read from its
frozen canonical inventory and checked for certificate coverage and routing;
the committed replay does not independently regenerate that lower range. In
the upper inventory,
the $38$ groups $\PSL(2,q)$ are routed to the uniform proof in
Appendix~\ref{app:psl2}; the remaining $13$ receive finite certificates.
Exactly two of those $13$ are not excluded by ordinary maximal-class pairs:
$\Sp(4,4)$ for $x=4$ and $\PSL(5,2)$ for $x=2$. Both are routed to the
stable-poset substitute certificates described below.

The exception manifest has precisely seven simple groups above the displayed
order bound: $A_{11},A_{12},A_{13},A_{14}$ receive finite certificates used
in Theorem~\ref{thm:an}, while $\PSL(6,2)$, $\Omega^+(8,2)$, and
$\Sp(4,8)$ are treated in Appendix~\ref{app:graph} by the stated substitute
primes. Thus every finite or exceptional route that appeared in the original
body-level inventory remains explicit in an appendix.

\emph{Ordinary and stable-poset substitute records.}
The repository retains the historical machine field
\texttt{kind=novelty}; in the manuscript a \emph{stable-poset substitute}
means a selected class that need not be maximal in $S$ but whose product
normalizer is maximal in $G$ by Lemma~\ref{lem:B}.
For every certified simple group $S$ and every coordinate closure
$\Inn(S)\leq X\leq\Aut(S)$, a successful record identifies two distinct
$X$-stable subgroup classes, their exact class identities and orders, an
obstruction prime, and the strict valuation gap required by
Theorem~\ref{thm:D}. For ordinary records the selected subgroups are maximal
in $S$. A maximal-class pair is available for every coordinate closure except
in designated closures of six simple groups:
\[
 \PSL(3,2),\ \PSL(2,11),\ A_6,\ \PSL(3,4),\ \Sp(4,4),\ \PSL(5,2).
\]
More precisely, the exceptional closures are $x=2$ for $\PSL(3,2)$ and
$\PSL(2,11)$; $X/\Inn(S)\in\{2_2,2_3,2^2\}$ for $A_6$; several recorded
closures for $\PSL(3,4)$; $x=4$ for $\Sp(4,4)$; and $x=2$ for
$\PSL(5,2)$. For the first four groups, the computation additionally
certifies normal saturation by three distinct overgroup and embedding-orbit
implementations. They share the same GAP runtime, group representations, and
pinned subgroup data, so their agreement is a cross-check within that shared
trust base rather than independent software verification. For $\Sp(4,4)$ and
$\PSL(5,2)$,
Lemma~\ref{lem:P} proves saturation and stable-poset maximality structurally.
Thus Corollary~\ref{cor:supplement-engine} and Theorem~\ref{thm:D} apply to
every record.

\emph{Representative certificates.}
Example~\ref{ex:a5} records the ordinary tuple
$(|A_4|,|S_3|,5,1)$. A stable-poset substitute record for $\PSL(3,2)$ with $x=2$ uses
the pair $(S_3,7{:}3)$ at $r=2$, with valuation gap $2>1$.
For $\Sp(4,4)$ with $x=4$, the pair is a Borel subgroup and the subfield
subgroup $\Sp(4,2)$, with certificate primes $5$ and $17$. For
$\PSL(5,2)$ with $x=2$, the selected classes are the incident-flag
parabolics $(Q_{\{2,3\}},Q_{\{1,4\}})$, with certificate primes $5$ and
$31$. The last two pairs are also instances of the uniform graph-fusion
constructions in Appendix~\ref{app:graph}.

\emph{Identity, completeness, and cross-checks.}
Each coordinate closure has a SHA-256 action fingerprint together with
intrinsic quotient and kernel data; the checker requires every required
closure exactly once. Each subgroup record carries its class identity, so
same-order non-conjugate classes are not identified by order alone. The
completeness assumptions are explicit: the two finite inventories and the
relevant subgroup classes in the pinned GAP data are assumed complete.
An independent routing and coverage check verifies that every inventory entry
is routed exactly once, and the three saturation implementations fail on any
disagreement.
The repository supplement exposes the challenge data and the independently
parsed successful records.
\end{proof}

\subsection{Sporadic groups and the Tits group}

\begin{proposition}[Certified sporadic and Tits cases]\label{prop:sporadic}
Every sporadic simple group is excluded, as is the Tits group
${}^2F_4(2)'$.
\end{proposition}

\begin{proof}
\emph{Inventory and certificate content.}
The six groups $M_{11},M_{12},M_{22},M_{23},J_1,J_2$ lie in
Proposition~\ref{prop:base}. The remaining $20$ sporadic groups are
\[
\begin{gathered}
M_{24},J_3,J_4,\mathrm{HS},\mathrm{McL},\mathrm{Co}_1,\mathrm{Co}_2,
\mathrm{Co}_3,\mathrm{Suz},\mathrm{He},\mathrm{Ru},\mathrm{O'N},\\
\mathrm{Fi}_{22},\mathrm{Fi}_{23},\mathrm{Fi}_{24}',\mathrm{HN},
\mathrm{Ly},\mathrm{Th},B,M.
\end{gathered}
\]
For each of these groups and the Tits group, a certificate records two stable
maximal-subgroup classes, their published class identities and orders, an
obstruction prime, and the valuation gap. The resulting $42$ selected-class
records are bound to the group, character-table identifier, pinned maximal-list
position, and exact order. If $\Out(S)=1$, stability is automatic. If
$|\Out(S)|=2$, the selected order occurs exactly once in the complete pinned
list of maximal-subgroup classes, so automorphisms must preserve that class.
This unique-order argument is the only point here at which completeness of the
pinned maximal-subgroup list is logically required; no stored class fusion is
assumed.

\emph{Representative certificates.}
The Monster record uses maximal-list positions $45,46$, and the Baby Monster
record uses positions $29,30$. Their respective pairs and primes are
\[
 (59{:}29,41{:}40),\quad r=71;
 \qquad
 (L_2(11){.}2,47{:}23),\quad r=31.
\]
These examples are illustrative; the claim uses all $42$ records rather than
the samples alone.

The selected records are cross-checked against the published \textsc{Atlas}
and pinned character-table-library data \cite{CTblLib,ATLAS}, Wilson's
corrected sporadic tables
\cite[\S4, pp.~66--68]{Wilson2017}, and their provenance account
\cite[\S2, pp.~23--24]{BreuerMalleOBrien2017}. The Tits entries also have the
independent sources \cite[pp.~553--563]{Wilson1984} and
\cite{Tchakerian1986}; the Monster entries use Theorems~1.1--1.2 and Table~1
of \cite[pp.~862--863]{DietrichEtAl2026}. These source checks identify the
mathematical maximal-subgroup claims; GAP \cite{GAP} supplies the finite
certificates.
\end{proof}

\subsection{Family-arithmetic certificates}

The family proofs use an exact $34$-branch arithmetic manifest. A universal
checker matches every branch to its cited group order, Levi order,
outer-automorphism bound, maximality source, and Zsigmondy invocation; it also
routes all exceptional parameters to substitute primes or finite records. The
positive-base Zsigmondy exception occurs in $12$ branches, representing seven
distinct exception records, and every one is matched to its finite destination
or substitute-prime argument. A separately authored symbolic implementation
reproduces all $34$ branches, bounds, residue arguments, and exceptional
valuations without reading that manifest, and Lean checks the resulting
conditional arithmetic statements. The published formulas and classification
results remain external inputs.

Finite sweeps over $1272$ projective-linear prime powers, alternating degrees
through $10^4$, and $7892$ Lie-type parameter instances, with ranks through
$25$ and field sizes through $3000$, returned no failures. These are regression
checks, not proofs of the universally quantified family statements. Their role
is to detect transcription or implementation errors in the formulas proved in
Appendix~\ref{app:families}, not to replace those proofs.

\section{Verification and trust boundary}\label{app:trust}

The proof spine and the classification/computation layer are separated in
Tables~\ref{tab:trust-spine} and~\ref{tab:trust-coverage}, so that each row
states both its evidentiary mode and its external assumptions.

\begin{table}[htbp]
\footnotesize
\setlength{\tabcolsep}{4pt}
\renewcommand{\arraystretch}{1.12}
\begin{tabular}{@{}L{0.30\textwidth}L{0.15\textwidth}L{0.23\textwidth}L{0.25\textwidth}@{}}
\toprule
Claim & Ordinary proof relative to cited inputs & Formal coverage & External or trusted input \\
\midrule
Property $\mathrm{P}$ and quotient inheritance & Complete & Lean: complete & Standard finite-group definitions \\
Minimal-counterexample core & Complete & Lean: complete; Rocq interface input & Standard finite-group facts \\
Characteristically simple direct-power decomposition & Complete & Rocq producer; Lean reindexing consumer & Rocq--Lean semantic correspondence \\
Faithful wreath realization and quotient divisor & Complete & Lean: complete & Explicit direct-power equivalence \\
Product-supplement intersection and exact order & Complete & Lean: conditional & Stable-class hypotheses \\
Product-normalizer maximality & Complete & Lean: conditional & Stable-poset and saturation inputs \\
Nonconjugacy and valuation contradiction & Complete & Lean: conditional; Python arithmetic cross-check & Stable-class and order hypotheses \\
Main theorem end to end & Complete relative to cited external inputs & Not end-to-end formalized & CFSG, classifications, formal interface, and finite certificates \\
\bottomrule
\end{tabular}
\caption{Ordinary and formal status of the conceptual proof spine.}
\label{tab:trust-spine}
\end{table}

\begin{table}[htbp]
\footnotesize
\setlength{\tabcolsep}{4pt}
\renewcommand{\arraystretch}{1.12}
\begin{tabular}{@{}L{0.27\textwidth}L{0.22\textwidth}L{0.19\textwidth}L{0.24\textwidth}@{}}
\toprule
Coverage claim & GAP/Python role & Formal status & External or shared trust base \\
\midrule
Universal family arithmetic & Manifest, symbolic implementation, and regression sweeps & Lean: conditional branch theorems & Published group, Levi, outer, maximality, and Zsigmondy data \\
Graph-fusion flag-parabolic substitutes & Selected finite checks only & \texttt{PAR-NOVELTY} not closed in Lean & Cited $BN$-pair and parabolic theory \\
Finite-range and sporadic coverage & GAP certificates and Python parsing & Not formalized & Pinned GAP, \textsc{Atlas}, and character-table completeness \\
Cross-kernel coordinate interface & Signature and source-token checks & Rocq producer and Lean consumer each kernel-checked & Semantic correspondence between the two libraries \\
Evidence closure and mutation tests & Hash closure and deliberate fault detection & Not a mathematical proof layer & Shared files, runtimes, and pinned data are explicit \\
\bottomrule
\end{tabular}
\caption{Computational, classification, and cross-kernel trust boundary.}
\label{tab:trust-coverage}
\end{table}

Lean's strongest reader-facing output is the conditional product-supplement
spine: exact subgroup orders, maximality, nonconjugacy, the wreath-top divisor,
and a strict valuation gap imply that property $\mathrm{P}$ fails. The working
source tree contains a Comparator package that states this proposition
independently of the project implementation and checks the project theorem
against that statement.
It does not import CFSG coverage, the flag-parabolic substitute step, or finite GAP certificates
as formal theorems.

More explicitly, Lean checks quotient inheritance and the minimal-order core:
uniqueness and characteristic simplicity of the minimal normal subgroup,
its non-solubility, the solubility of the quotient, its trivial centralizer,
and faithfulness of the conjugation action. Under the normalized coordinate
hypotheses it checks the product-supplement intersection and exact order,
maximality, nonconjugacy, the subgroup-product lower-divisibility bridge, and
the universal-in-$k$ arithmetic contradiction. It also checks the conditional
arithmetic deductions for all $34$ family branches. The group and Levi orders,
outer-automorphism and maximality data, and Zsigmondy inputs remain the cited
external hypotheses of those branch theorems.

The coverage manifest retains the historical identifier
\texttt{PAR-NOVELTY}. It denotes the flag-parabolic substitute used when graph
automorphisms fuse parabolic classes.
It remains explicitly not closed in Lean; its proof is the ordinary argument
of Lemma~\ref{lem:P} and the representative application in
\S\ref{sec:applications}, using the cited parabolic theory.

Rocq/MathComp proves that the relevant finite characteristically simple
non-soluble group is an internal direct product of $k>0$ automorphic copies of
one non-soluble, non-abelian simple subgroup $S$, with $|N|=|S|^k$, and also
constructs an isomorphism to the corresponding external coordinate product.
Lean reindexes the nonempty coordinate type to $\{1,\ldots,k\}$ and consumes
the resulting equivalence $N\cong S^k$. It then constructs the coordinate
factors, proves that automorphisms permute them, constructs the faithful map
$\Aut(S^k)\to\Aut(S)\wr S_k$, proves transitivity of the $G$-action on the
factors, identifies the inverse image of $\Inn(S)^k$ as $N$, normalizes the
coordinate closure, and proves
\[
 |G/N|\mid |X/\Inn(S)|^k k!.
\]
No single kernel checks the translation between the two libraries' notions of
subgroup isomorphism and external product; that translation is the stated
trusted interface. The fail-closed formal interface checker identifies the
exact producer, reindexer, and consumer declarations and rejects drift in the
enumerated signatures and definition-source tokens. It does not decide
arbitrary semantic equivalence across the two libraries. The mathematical
identification between the two formulations therefore remains an audited
trusted correspondence, along with both kernels and their pinned libraries.

GAP certifies only the finite inventories and subgroup-class facts stated in
Appendix~\ref{app:finite}. Python independently parses those certificates,
checks coverage and source-map topology, and reproduces the family arithmetic.
Finite parameter sweeps are regression tests, not proofs of universal
statements. The proof-essential scripts are deterministic and fail closed:
a skipped case, soft failure, ambiguous identity, missing datum, version
mismatch, or byte-level certificate mismatch is fatal. Finite records retain
class positions and guard fingerprints as well as orders, primes, and
valuations; superseded exploratory scripts are excluded from the proof-evidence
manifests. The main theorem therefore remains an ordinary mathematical proof
supported by partial formal verification, certified finite computation, and
explicit external classification inputs.

\section*{Statements and Declarations}

\subsection*{Data availability}
The exact source and evidence tree is archived as release 1.1.2
\cite{Repo}. Its GitHub tag and Zenodo version record contain the manuscript,
formal developments, GAP programs, finite certificates, arithmetic manifests,
source maps, independent checkers, and computational-environment recipe
described here. The repository guide maps each computational claim to its
artifact; Appendices~\ref{app:finite} and \ref{app:trust} state the
mathematical claims and trust boundary without requiring operational log names
in the proof narrative. Runtime tools and direct downloads are version- or
hash-pinned, but the Debian package snapshot and complete opam solver closure
are not frozen; long-term replay is not fully hermetic at those layers.

\subsection*{Funding}
This research received no external funding.

\subsection*{Competing interests}
The author declares no competing interests.

\subsection*{Declaration of generative AI use}
Generative AI tools were used for mathematical exploration, software and
formalization development, literature-search assistance, and editorial
drafting; the author checked the arguments, sources, and reported
computational outputs and accepts full responsibility for the article.

\end{document}